\documentclass{amsart}

\usepackage{amssymb,amsbsy}
\usepackage{enumerate}

\usepackage{amsrefs}

\newtheorem{theorem}{Theorem}[section]
\newtheorem{lemma}[theorem]{Lemma}
\newtheorem{proposition}[theorem]{Proposition}
\newtheorem{corollary}[theorem]{Corollary}

\theoremstyle{definition}
\newtheorem{definition}[theorem]{Definition}

\theoremstyle{remark}
\newtheorem{remark}[theorem]{Remark}
\numberwithin{equation}{section}
\newtheorem*{theorem*}{Theorem}
\newcommand{\field}{\boldsymbol{k}}

\newcommand{\N}{\mathbb N}

\newcommand{\Prim}{\mathop{\mathrm{Prim}}}
\newcommand{\PBW}{\mathrm{PBW}}

\newcommand{\sset}{\mathfrak s}
\newcommand{\T}{\mathbb T}
\newcommand{\Z}{\mathbb Z}

\newcommand{\1}{_{(1)}}
\newcommand{\2}{_{(2)}}
\newcommand{\3}{_{(3)}}
\newcommand{\4}{_{(4)}}

\DeclareMathOperator{\BCH}{\mathrm{BCH}}
\DeclareMathOperator{\Der}{\mathrm{Der}}

\DeclareMathOperator{\Endo}{\mathrm{End}}

\DeclareMathOperator{\Hom}{\mathrm{Hom}}
\DeclareMathOperator{\Id}{\mathrm{Id}}

\DeclareMathOperator{\spann}{\mathrm{span}}

\author{A.  Grishkov}
\address{Departamento de Matem\'atica e Estat{\'\i}stica, Universidade de S\~ao Paulo, S\~ao Paulo, SP 05311-970, Brazil}
\email{shuragri@gmail.com}

\author{J. M. P\'erez-Izquierdo}
\address{Departamento de Matem\'aticas y Computaci\'on, Universidad de
La Rioja, 26004 \\ Lo\-gro\-\~no, Spain}
\email{jm.perez@unirioja.es}
\thanks{J.M. P\'erez-Izquierdo  thanks support from the Spanish Ministerio de Ciencia e Innovaci\'on (MTM2013-45588-C3-3-P) and the Programa hispano--brasile\~no de cooperaci\'on interuniversitario (PHBP14/00110). He also thanks I. Shestakov and J. Mostovoy for encouragement.}%

\keywords{Loops, Sabinin algebras, Commutative $A$--loops, Lie triple systems}
\subjclass[2010]{20N05,17A40}
\begin{document}
\title{Lie's Correspondence for Commutative Automorphic Formal Loops}
\begin{abstract}
We develop Lie's correspondence and an explicit Baker-Campbell-Hausdorff formula for commutative automorphic formal loops. 
\end{abstract}
\maketitle
%
%
%
\section{Introduction}

Loops are the non-associative counterpart of groups. These algebraic structures have a product $xy$ and a unit element $e$. Apart from this, the only extra requirement is that the left and right multiplication operators $L_x \colon y \mapsto xy$ and $R_x \colon y \mapsto yx$ are invertible for all $x$, which is equivalent to the existence of left and right divisions $x \backslash y$ and $x/y$ such that $x \backslash (xy) = y = x ( x \backslash y)$ and $(yx)/x = y = (y/x)x$. The lack of associativity uncovers a tremendous rich `phylogenetic tree' of varieties of loops  that has motivated recent developments in non-associative mathematics. The present paper gives more evidence about the close relationship between commutative automorphic loops and abelian groups as `non-associative species'.

The interest in loops began in the 1930s with the work of Moufang on projective geometry. Loops that  satisfy $x(y(xz)) = ((xy)x)z$ are now called Moufang loops in her honor. In 1955 Malcev \cite{Ma55} noticed that Lie's approach to the study of local analytic groups might work even when associativity is relaxed. Under this new point of view Lie algebras are just the tangent algebras of associative analytic loops, but many other varieties of tangent algebras exist. Moufang analytic loops are diassociative--i.e., the subloop generated by any two elements is a group--and their tangent algebras, now called Malcev algebras, are binary-Lie algebras--i.e. the subalgebra generated by any two elements is a Lie algebra. Another interesting observation from Malcev was that the Baker-Campbell-Hausdorff formula only depends on two elements, thus the same formula makes sense for binary-Lie algebras. This suggested that finite-dimensional real Malcev algebras integrate to local analytic Moufang loops, as  proved in 1970 by  Kuzmin \cites{Ku70,Ku71}. Since then, the study of Lie's correspondence in  non-associative settings was a challenging problem (see for instance \cites{Ya57,Ak76,Ki75,MiSa82,MiSa85,HoSt86,Sm88,Sa90,Sa90a,Sa91,Fi00,NaSt02,Go06}), finally solved by Mikheev and Sabinin  in 1987 with the apparatus of affine connections from differential geometry \cite{MiSa87}. The tangent algebra of a local analytic loop is a Sabinin algebra--an algebraic structure with two infinite families of multilinear operations satisfying certain axioms--and,  under certain convergence conditions, any finite-dimensional real Sabinin algebra is the tangent algebra of a (uniquely determined up to isomorphism)  local analytic loop. 

While the result of Mikheev and Sabinin shows that Lie's correspondence remains valid even when associativity is removed, in practice it is difficult to compute the identities that define the varieties of Sabinin algebras associated to  varieties of loops. For instance, the theory ensures the existence of two infinite families of multilinear operations on the tangent space of any local analytic Moufang loop that classify it; however, in practice, only a binary operation is required since the other multilinear operations can be derived from this one, and the axioms satisfied by this binary operation do not clearly follow from those of Sabinin algebra. Thus a case-by-case approach is required in the study of varieties of loops.

Over the years several varieties of loops and quasigroups have been studied in connection with geometry (see the books \cites{Go88,Ch90,AkSh92,Sa99,AkGo00} and references therein) and new examples of Lie's correspondence have appeared. Recently, the variety of automorphic loops introduced in 1956 by Bruck and Paige \cite{BP56} has attracted a lot of attention and it is an active area of research (see \cite{Vo15} for an updated account on the subject). These loops are defined by the following property: \emph{the stabilizer of the unit element $e$ in the group generated by the left and right multiplication operators consist of automorphisms of the loop}. For groups the elements of this stabilizer are nothing else but the usual adjoint maps $L_a R^{-1}_a$ that, as any undergraduate student knows, are automorphisms. However, the study of automorphic loops is technically quite demanding and  advances in this subject initially required computer assistance. In \cite{JKV11} Jedlicka, Kinyon and Vojt{\v{e}}chovsk{\'y} proved that any commutative automorphic loop of odd order is solvable. Later \cite{JKNV11} Johnson, Kinyon, Nagy and Vojt{\v{e}}chovsk{\'y} initiated a search of simple commutative automorphic loops of small order with the help of GAP. In \cite{Na14} Nagy studied commutative automorphic loops of exponent $2$ by means of Lie rings. Finally, Grishkov, Kinyon and Nagy \cite{GKN14} proved that any commutative automorphic finite loop is solvable. In \cite{KKPV12} the same result has been obtained for automorphic finite loops of odd order.

In this paper we would like to advance towards the understanding of local analytic commutative automorphic loops. Since the paper is targeted to algebraists, to avoid the use of differential geometry we will work with commutative automorphic formal loops and our techniques will rely on non-associative Hopf algebras.

Tangent algebras of commutative automorphic formal loops will be called commutative automorphic Lie triple systems. These are vector spaces $T$ equipped with a linear triple product $[\,,\,,\,] \colon T \otimes T \otimes T \rightarrow T$ such that
\begin{enumerate}
\item $[a,b,c] = - [b,a,c]$,
\item $[a,b,c]+[b,c,a]+[c,a,b] = 0$ and
\item  $[[a,b,c],a',b'] =[[a,a',b'],b,c] + [a,[b,a',b'],c] + [a,b,[c,a',b']]$. 
\end{enumerate}
for any $a,b,c, a', b' \in T$. Our main result is the following Lie's correspondence: \emph{over fields of characteristic zero, the category of commutative automorphic formal loops is equivalent to the category of commutative automorphic Lie triple systems.} We will also derive an explicit  Baker-Campbell-Hausdorff formula for commutative automorphic loops:
\begin{displaymath}
\BCH(a,b)^{\cdot} = a + b + \sum_{i,j\geq 1} \beta_{i,j} [a,b,\stackrel{i-1}{\dots}a,\stackrel{j-1}{\dots}b]
\end{displaymath}
where $\beta_{i,j}$ ($i,j \geq 1$) is the coefficient of $s^it^j$ in the Taylor expansion of
\begin{displaymath}
 \frac{\left(e^{2 s}-e^{2 t}\right) (s+t)}{2 \left(e^{2 (s+t)}-1\right)}
\end{displaymath}
at $(0,0)$,
\begin{equation*}
[a_1,a_2,\dots,a_n] :=
\left\{
\begin{array}{ll}0 & \text{if } n \text{ is even} \cr
a_1 & \text{if } n= 1\cr
[[[a_1,a_2,a_3],\cdots],a_{n-1},a_{n}]&  \text{if } n > 1 \text{ is odd}
\end{array}\right.
\end{equation*}
and
\begin{displaymath}
\stackrel{i}{\dots}c := c,c,\dots,c \quad \text{where } c \text{ appears } i \text{ times.}
\end{displaymath}
We hope that this formula will be useful for researches working on finite loops. In \cite{Na02} a Baker-Campbell-Hausdorff formula was studied for Bruck loops. 

The paper is structured as follows. In Section~\ref{sec:cafl} we recall the definition of formal loop and the relationship between commutative automorphic loops and left Bruck loops, from which the algebraic structure on the tangent space of any  commutative automorphic formal loop is easily derived. We study this algebraic structure in Section~\ref{sec:calts}. Special attention is paid to the commutative automorphic Lie triple system freely generated by two generators since we will be concerned with a Baker-Campbell-Hausdorff formula. This triple also plays an important role in our proof on the formal integration of commutative automorphic Lie triple systems. This proof occupies Section~\ref{sec:fi}. An explicit Baker-Campbell-Formula is presented in Section~\ref{sec:bch}.

\subsection{Notation.} In this paper the characteristic of the base field is assumed to be zero. If not explicitly established otherwise, $[x,y]$ and $(x,y,z)$ will stand for the commutator $xy - yx$ and the associator $(xy)z - x(yz)$ respectively. We will stick to the following order of parentheses for powers: 	$x^n := x(x(\cdots (xx)))$. Finally, coalgebras are always assumed to be cocommutative and coassociative even when not explicitly mentioned.

\section{Commutative automorphic formal loops}
\label{sec:cafl}
\subsection{Formal loops and non-associative Hopf algebras. Linearization.}
\subsubsection{Hopf algebras of symmetric powers.} Let $T$ be a vector space over a field $\field$ of characteristic zero and let $\field[T]$ be the symmetric algebra on $T$ with product $xy$. The maps
\begin{displaymath}
\Delta(a) = a \otimes 1 + 1 \otimes a \quad \text{and} \quad \epsilon(a) = 0 \quad (a \in T)
\end{displaymath}
can be uniquely extended to homomorphisms of unital algebras\footnote{To achieve conciseness we will omit the symbol $\sum$  in Sweedler's notation $ \sum x\1 \otimes x\2$ for $\Delta(x)$.}
\begin{eqnarray*}
\Delta \colon \field[T]  &\rightarrow& \field[T]  \otimes \field[T]\\
x &\mapsto& x\1 \otimes x\2 
\end{eqnarray*}
and 
\begin{eqnarray*}
\epsilon \colon \field[T] &\rightarrow& \field \\
x &\mapsto& \epsilon(x)
\end{eqnarray*}
so that $(\field[T], \Delta, \epsilon)$ is a \emph{coassociative and cocommutative coalgebra}. Cocommutativity refers to the property $x\1 \otimes x\2 = x\2 \otimes x\1$ while coassociativity means $(\Delta \otimes \Id) \Delta =  ( \Id \otimes \Delta)\Delta$, i.e. $({x\1}\1 \otimes {x\1}\2) \otimes x\2  = x\1 \otimes ({x\2}\1 \otimes {x\2}\2)$. Hence, there is no ambiguity in writing $x\1 \otimes x\2 \otimes x\3$ for $(\Delta \otimes \Id) \Delta(x)$, or more generally  $x\1 \otimes x\2 \otimes \cdots \otimes x_{(n+1)}$ for the image of $x$ after applying  $\Delta$ $n$ times. The \emph{unit}  is the map $u \colon \field \rightarrow \field[T]$ given by $\alpha \mapsto \alpha 1$. The algebraic structure $(\field[T], xy, u, \Delta,\epsilon)$ is a \emph{commutative and associative connected  bialgebra}. If we also include the $1$-ary operation given by the \emph{antipode}, i.e. the automorphism $S$ of $\field[T]$ induced by $a \mapsto -a$ for any $a \in T$, then we get a  commutative connected \emph{Hopf algebra} (the theory of coalgebras, bialgebras and Hopf algebras can be found in \cites{Ab80,Sw69} for instance).  However, the product in this Hopf algebra structure on $\field[T]$  is irrelevant for us since it  corresponds to  abelian formal groups instead of  commutative automorphic formal loops. We will  keep the coalgebra structure and the unit on $\field[T]$ but we will consider some new non-associative products on $\field[T]$.

\subsubsection{Non-associative Hopf algebras.} In this paper a \emph{(non-associative) Hopf algebra} $(H, m, u, \backslash, /, \Delta, \epsilon)$ refers to a cocommutative and coassociative coalgebra $(H, \Delta, \epsilon)$ endowed with the following linear maps: a \emph{product} $m \colon H \otimes H \rightarrow H$, a \emph{unit} $u \colon \field \rightarrow H$, a \emph{left division} $\backslash \colon H \otimes H \rightarrow H$ and a \emph{right division} $/\colon H \otimes H \rightarrow H$ so that
$\Delta(xy) = \Delta(x) \Delta(y)$, $\Delta(1) = 1 \otimes 1$, $\epsilon(xy) = \epsilon(x) \epsilon(y)$,  $\epsilon(1) = 1$ and
\begin{align*}
x\1 \backslash (x\2 y) &= \epsilon(x) y = x\1 (x\2 \backslash y)\\
 (y x\1) / x\2 &= \epsilon(x) y = (y / x\1) x\2
\end{align*}
where $xy := m(x \otimes y)$ and $1 := u(1)$ is the unit element (see \cite{MPS14} for a survey on non-associative Hopf algebras). In case that $H$ is associative then the left and right divisions are $x \backslash y = S(x) y$ and $x/y = xS(y)$ where $S$ is the antipode. However, non-associative Hopf algebras lack of antipode in general. 
\subsubsection{Connected Hopf algebras.}  Hopf algebras with coalgebra structure isomorphic to $(\field[T], \Delta, \epsilon)$ for some vector space $T$ are called \emph{connected} (see \cite{Sw69} for the precise definition). For these Hopf algebras the left and right division can be easily derived from the product.  For instance $1 \backslash (1y) = y$ implies that $ 1 \backslash y = y$. For elements $a \in T$ we have
\begin{displaymath}
a \backslash (1y)  + 1 \backslash(ay) = \epsilon(a) y = 0 \text{ thus } a \backslash y = - ay,
\end{displaymath}
etc. Connected Hopf algebras are much more friendly than general Hopf algebras since many maps can be constructed recursively in this way. We will use this feature several times.
\subsubsection{Primitive elements.} 
Elements $a$ in a Hopf algebra $H$ such that 
\begin{displaymath}
\Delta(a) = a \otimes 1 + 1 \otimes a
\end{displaymath}
are called \emph{primitive}. The subspace of all primitive elements of $H$ is denoted by $\Prim(H)$. Shestakov and Umirbaev \cite{ShUm02} realized that this space admits many algebraic operations that generalize the usual Lie product on the tangent space of local analytic groups. With these operations $\Prim(H)$ is a Sabinin algebra. Thus $\Prim(H)$ will play an important role in this paper since it can be understood as the tangent algebra of formal loops. 
\subsubsection{Products $xy$ and $x \cdot y$.}  In this paper Hopf algebras are non-associative, so we will omit this adjective; they will satisfy some identities, but associativity is related to formal groups rather than to more general formal loops. In this context  several products naturally appear on the same coalgebra to give different Hopf algebras. We will be concerned with those Hopf algebras related to commutative automorphic loops but they will be obtained  from other Hopf algebras, with the same underlying vector spaces, linked with left Bruck loops. To distinguish between both structures, we will use $x \cdot y, x \dot{\backslash} y, x \dot{/} y$ for the former (commutative automorphic Hopf algebras or loops) and $xy, x \backslash y, x/y$ for the latter (left Bruck Hopf algebras or loops). Beware, none of these structures is  the natural commutative and associative Hopf algebra structure on  $\field[T]$. Since the coalgebra structure is fixed we will refer to Hopf algebras without any mention to the coalgebra structure or to the unit.

\subsubsection{Formal loops.} A \emph{formal loop} is a map $F \colon \field[T] \otimes \field[T] \rightarrow T$, also denoted by $\mathbf{x}\mathbf{y}$, that satisfies
\begin{displaymath}
F\vert_{\field[T] \otimes 1} = \pi_T = F\vert_{1 \otimes \field[T]}
\end{displaymath} 
where $\pi_T$ stands for the natural projection of $\field[T]$ onto $T$. The distinguished role of $\pi_T$ in this context is linked to eulerian idempotents. Sometimes we refer to the pair $(\field[T], F)$ or $(\field[T], \mathbf{x}\mathbf{y})$ as the formal loop. The map $F$ is modeled after the Taylor expansion of local loops and it codifies all the information required to construct a connected Hopf algebra. Any such Hopf algebra defines a formal loop by $F(x \otimes y) := \pi_T(xy)$. 
\subsubsection{The Hopf algebra of formal distributions of a formal loop.}
Any formal loop $(\field[T],F)$ can be extended to a map $ m \colon \field[T] \otimes \field[T] \rightarrow \field[T]$
given rise to a new Hopf algebra structure (see \cite{MP10}):
\begin{displaymath}
 \field[F] := (\field[T], m , u,  \backslash, /, \Delta, \epsilon)
\end{displaymath}
on the coalgebra $(\field[T], \Delta, \epsilon)$--recall that the left and right divisions can be derived from the product. This Hopf algebra is called the \emph{Hopf algebra (or bialgebra) of formal distributions with support at the identity of the formal loop $F$}.

With independence of the formal loop, as a unital algebra $\field[F]$ is always generated by $T$. In fact, $\field[F]$ is filtered by the powers of $\ker \epsilon$, and the corresponding graded algebra is isomorphic to the symmetric algebra $\field[T]$ (Poincar\'e-Birkhoff-Witt Theorem \cites{MP10,PI07}). The correspondence 
\begin{displaymath}
 F \mapsto \field[F]
\end{displaymath}
between formal loops and connected Hopf algebras is an equivalence of categories \cite{MP10}. Therefore the study of formal loops is equivalent to the study of non-associative Hopf algebras. 

\subsubsection{Commutative automorphic and left Bruck formal loops.} Identities such as $(\textbf{x}\textbf{y})\textbf{z} = \textbf{x}(\textbf{y}\textbf{z})$ make sense for formal loops but the reader should consult \cite{MP10} for the rigorous interpretation of them since these expressions are just a way of avoiding the cumbersome occurrence of the comultiplication in the identities satisfied by $\field[F]$. Identities on formal loops are not required in this paper so our advice is to focus on identities on non-associative Hopf algebras.

\begin{definition}
\label{def:CA}
A formal loop $(\field[T],\mathbf{x} \cdot \mathbf{y})$ is \emph{commutative automorphic} if it satisfies the identities:
\begin{enumerate}
\item (\emph{commutative}) $\mathbf{x} \cdot \mathbf{y} = \mathbf{y} \cdot \mathbf{x}$  and
\item (\emph{left automorphic})
\begin{eqnarray*}
&& (\mathbf{x} \cdot \mathbf{y}) \dot{\backslash}(\mathbf{x} \cdot (\mathbf{y} \cdot (\mathbf{w} \cdot \mathbf{z}))) \\
&& \quad\quad = \left((\mathbf{x} \cdot \mathbf{y})\dot{\backslash}(\mathbf{x} \cdot (\mathbf{y} \cdot \mathbf{w}))\right) \cdot \left((\mathbf{x} \cdot \mathbf{y})\dot{\backslash}(\mathbf{x} \cdot (\mathbf{y} \cdot \mathbf{z}))\right).
\end{eqnarray*}
\end{enumerate}
Equivalently, the formal loop $(\field[T],\dot{F})$ is \emph{commutative automorphic} if $\field[\dot{F}]$ is a \emph{commutative automorphic Hopf algebra}, i.e. it satisfies the identities:
\begin{enumerate}
\item (\emph{commutative}) $x \cdot y = y \cdot x$ and
\item (\emph{left automorphic}) 
\begin{eqnarray*}
&&  (x\1 \cdot y\1 ) \dot{\backslash} (x\2 \cdot (y\2  \cdot (w \cdot z)))\\
 && \quad\quad  =
 \quad \left((x\1 \cdot y\1)\dot{\backslash} (x\2  \cdot (y\2 \cdot w))\right) \cdot \left((x\3 \cdot y\3)\dot{\backslash} (x\4 \cdot (y\4 \cdot z))\right).
\end{eqnarray*}
\end{enumerate}
The second identity will be written as
\begin{displaymath}
\dot{l}(x,y)(w \cdot z) = \dot{l}(x\1,y\1)(w) \cdot \dot{l}(x\2,y\2)(z)
\end{displaymath}
with 
\begin{displaymath}
\dot{l}(x,y)(w) := (x\1 \cdot y\1)\dot{\backslash} (x\2  \cdot (y\2 \cdot w)).
\end{displaymath}
\end{definition}

\begin{definition} A formal loop $(\field[T],\mathbf{x} \mathbf{y})$ is \emph{left Bruck} if it satisfies the identities:
\begin{enumerate}
\item (\emph{left Bol}) $\mathbf{x}(\mathbf{y}(\mathbf{x}\mathbf{z})) = (\mathbf{x}(\mathbf{y}\mathbf{x}))\mathbf{z}$ and
\item (\emph{automorphic inverse property}) $S(\mathbf{x}\mathbf{y})=S(\mathbf{x})S(\mathbf{y})$
\end{enumerate}
where $S(\mathbf{x}):=\mathbf{x}\backslash \mathbf{1}$.
Equivalently, the formal loop $(\field[T],F)$ is left Bruck if $\field[F]$ is a \emph{left Bruck Hopf algebra}, i.e. it satisfies the identities
\begin{displaymath}
x\1  (y  (x\2  z)) =(x\1 (y  x\2)) z \quad \text{and} \quad  S(x  y) = S(x)  S(y)
\end{displaymath}
for any $x,y$ and $z$, where $S(x):= x {\backslash} 1$.
\end{definition}
Commutative automorphic and left Bruck loops are defined by the same identities as their formal counterparts.

\subsubsection{Linearization.} We can \emph{linearize} identities for loops to obtain identities for coassociative and cocommutative Hopf algebras \cite{PI07} by replacing any repeated occurrence of any variable, say $x$, with $x\1$, $x\2$, etc. The occurrence of $x$ in only one side of the initial identity has to be corrected on the other side by multiplying that side by $\epsilon(x)$ to keep both sides of the identity being multilinear in all the variables. For instance, the identities that define commutative automorphic Hopf algebras or left Bruck Hopf algebras are obtained in this way from the identities that define the varieties of commutative automorphic loops or left Bruck loops respectively. If a new identity is consequence of the identities that define the given variety of loops then the linearization of that identity is  consequence of the linearization of the identities that define that variety \cite{PI07}. For instance, the variety of associative loops is the variety of groups. In any group the identities $x\backslash{1} = 1/x$,  $(x\backslash{1})(xy) = y = x ((x\backslash{1})y)$, $(yx)(1/x) = y = (y(1/x))x$ and $(xy) \backslash 1 = (y \backslash 1)(x \backslash 1)$ hold. In other words, if $S(x):=x^{-1}:=x \backslash{1} $ then $x\backslash y = S(x)y$ and $y/x = y S(x)$, thus instead of the binary operations $\backslash$ and $/$ we only consider the $1$-ary inverse map $S$. The same remains true for associative, coassociative and cocommutative Hopf algebras, where $x\backslash y$ and $x/y$ are superseded by the antipode $S(x) := x \backslash 1$. For instance, the identity $x\backslash{1} = 1/x$ for associative loops follows from
\begin{displaymath}
  x \backslash 1 =   ((1/x)x)(x \backslash 1) = (1/x)(x(x \backslash 1)) =  1/x.
\end{displaymath}
Linearizing these equalities we get 
\begin{displaymath}
x \backslash 1 =   ((1/x\1)x\2)(x\3 \backslash 1) = (1/x\1)(x\2(x\3 \backslash 1)) =  1/x
\end{displaymath}
for associative, coassociative and cocommutative Hopf algebras. Linearizing some other  identities on associative loops leads to the axioms for the antipode in the definition of (associative) Hopf algebras. Beware, in groups $S(S(x))=x$, hence the antipode of any associative, coassociative and cocommutative Hopf algebra must have order $2$, which is false in general if the hypotheses on coassociativity and cocommutativity fail. 

All Hopf algebras considered in this paper are coassociative and cocommutative so we will freely invoke this process of linearizing identities to obtain results for Hopf algebras from results on loops. Proposition \ref{prop:CA-LB} is the most important example where linearization is exploited.

\subsection{Commutative automorphic and left Bruck formal loops. Lie triple systems.}
\label{subsec:CAformalloos}

\subsubsection{The connection between commutative automorphic loops and left Bruck loops.} The study of commutative automorphic loops has experienced huge advances in recent years \cites{JKV10,JKV11,GKN14} due to the connections between these loops and left Bruck loops. 
This technique is useful in the context of formal loops too, so we will review very briefly some of the properties of commutative automorphic loops, left Bruck loops and the connection between these two varieties of loops. The following maps will appear frequently
\begin{displaymath}
S(x):=x^{-1}:= x \dot{\backslash} 1, \quad \dot{L}_x \colon y \mapsto x \cdot y, \quad\text{and}\quad  \dot{l}(x,y)\colon z \mapsto (x \cdot y) \dot {\backslash} (x \cdot (y \cdot z)).
\end{displaymath}
They are defined in terms of the $\cdot$ notation for the operations of commutative automorphic loops $(Q, 1, x \cdot y , x \dot{\backslash} y, x \dot{/} y)$ but they obviously have counterparts $S, L_x$ and $l(x,y)$ for left Bruck (or arbitrary) loops $(Q,1,xy, x \backslash y, x/y)$. 

\begin{proposition}[\cite{BP56}]
Every commutative automorphic loop and every left Bruck loop is power-associative, i.e. the subloop generated by any element is a group.
\end{proposition}

A wrong linearization might suggest that the subalgebra generated by any element in a commutative autormophic (or  left Bruck) Hopf algebra is associative (see (\ref{eq:Hopf-pa})). This is false and we are forced to make a choice regarding the powers of elements. In this paper we will stick to the following order of parentheses for powers:
\begin{displaymath}
x^n := x(x(\cdots (xx))).
\end{displaymath}

\begin{proposition}[\cite{JKV11}]
Commutative automorphic loops $(Q, x \cdot y)$ and left Bruck loops $(Q,xy)$ have the automorphic inverse property:
\begin{displaymath}
S(x \cdot y) = S(x) \cdot S(y) \quad \text{and} \quad S(xy) = S(x)S(y)
\end{displaymath}
for all $x, y \in Q$. In addition $S^2 = \Id$.
\end{proposition}
\begin{proposition}[\cite{Kr98}]
Every left Bruck loop $(Q,xy)$ satisfies the left automorphic property
\begin{displaymath}
l(x,y)(wz) = l(x,y)(w)l(x,y)(z)
\end{displaymath}
for all $x,y,w,z \in Q$.
\end{proposition}

\begin{proposition}[\cite{BP56}]
\label{prop:monoalternative}
Every left Bruck loop $(Q,xy)$ is left monoalternative:
\begin{displaymath}
x^m(x^ny) = x^{m+n}y
\end{displaymath}
for any $x,y\in Q$ and $n,m \in \Z$. In particular, $S(x)(xy) = y = x(S(x)y)$.
\end{proposition}

The fundamental connection between commutative automorphic loops and left Bruck loops is described in the next result. 

\begin{proposition}[\cite{JKV11}]
\label{prop:JKV11}
Let $(Q, x \cdot y)$ be a commutative automorphic loop and $P_{x} := \dot{L}^{-1}_{x^{-1}} \dot{L}_x = \dot{L}_x \dot{L}^{-1}_{x^{-1}}$. Then
\begin{displaymath}
P_x P_y P_x = P_{P_{x}(y)}
\end{displaymath}
for all $x,y \in Q$. Moreover, in case that $(Q, x \cdot y)$ is uniquely $2$-divisible--i.e. for any $x \in Q$ there exists a unique $\sqrt{x} \in Q$ such that $x = \sqrt{x} \cdot \sqrt{x}$--then the product 
\begin{displaymath}
xy := P_{\sqrt{x}}(y)
\end{displaymath}
defines a left Bruck loop structure on $Q$ and
\begin{displaymath}
x \cdot y = x \phi_x (y)
\end{displaymath}
with $\quad \phi_x := \dot{l}(\sqrt{x}^{-1},x)$.
\end{proposition}

\subsubsection{The connection between commutative automorphic and left Bruck formal loops.} Now we will discuss the linearization of the previous important results to obtain basic properties of commutative automorphic and left Bruck Hopf algebras. Let $(\field[T], \dot{F})$ be a commutative automorphic formal loop. Define on $\field[\dot{F}]$ maps 
\begin{align}
& S(x):= x^{-1} := x \dot{\backslash} 1, \nonumber \\ & \dot{l}(x,y) \colon z \mapsto (x\1 \cdot y\1) \dot {\backslash} (x\2 \cdot (y\2 \cdot z)), \nonumber\\
& r(x)  \text{ so that } x = r(x\1) \cdot r(x\2) \text{ and} \nonumber \\ & \phi_x := \dot{l}(S(r(x\1) ), x\2). \label{eq:defphi}
\end{align}
The map $r$ is recursively defined  on elements of $\field[\dot{F}]$. For instance $1 = r(1) \cdot r(1)$ implies $r(1) = 1$; $a = r(a) \cdot r(1)+r(1) \cdot r(a) = 2r(a)$ implies $r(a) = \frac{1}{2}a$ for any $a \in T$, etc.  

The linearization of the automorphic inverse property gives
\begin{equation}\label{eq:S}
S(x \cdot y) = S(x) \cdot S(y) 
\end{equation}
for all $ x,y \in \field[\dot{F}]$ and  
\begin{displaymath}
S^2 = \Id.
\end{displaymath}
We also have
\begin{equation}\label{eq:delta_r}
\Delta(r(x)) = r(x\1) \otimes r(x\2)
\end{equation}
and
\begin{equation}\label{eq:delta_phi}
\Delta(\phi_x(y)) = \phi_{x\1}(y\1) \otimes \phi_{x\2}(y\2).
\end{equation}
Both equalities are consequence of adequate linearizations although, for instance, to prove the former we also could observe that $\Delta(x) = \Delta (r(x\1)) \cdot \Delta(r(x\2))$ and $\Delta(x) = (r(x\1) \otimes r(x\2)) \cdot (r(x\3) \otimes r(x\4))$ and use that $\Delta(x) = \varphi(x\1) \cdot \varphi(x\2)$ uniquely determines the map $\varphi \colon \field[\dot{F}] \rightarrow \field[\dot{F}] \otimes \field[\dot{F}]$. The same kind of arguments leads to
\begin{displaymath}
\Delta (x\backslash y) = x\1 \backslash y\1 \otimes x\2 \backslash y\2, \quad \Delta(x/y) = x\1 / y\1 \otimes x\2 / y\2
\end{displaymath}
(see \cite{PI07} for a proof) from which (\ref{eq:delta_phi}) is a trivial consequence. Also notice that $\Delta(\phi_x(a)) = \phi_{x}(a) \otimes 1 + 1 \otimes \phi_{x}(a)$ for any $a \in T$ so
\begin{displaymath}
\phi_x(T) \subseteq T.
\end{displaymath}

Using the linearization of the left automorphic property we easily obtain
\begin{equation}
\label{eq:phiprod}
\phi_x( y \cdot z) = \phi_{x\1}(y) \cdot \phi_{x\2}(z)
\end{equation}
for all $x,y \in \field[\dot{F}]$. Since 
\begin{displaymath}
\Delta S = (S \otimes S) \Delta
\end{displaymath}
then
\begin{displaymath}
S \phi_x S = \phi_{S(x)}.
\end{displaymath}
Moreover, $S\phi_x S(1) = \epsilon(x) 1 = \phi_x(1)$ and for any $a \in T$, $S\phi_x S(a) = - S \phi_x(a) = \phi_x(a)$ because $\phi_x(T) \subset T$. As a unital algebra $\field[\dot{F}]$ is generated by $T$ so (\ref{eq:S}) and (\ref{eq:phiprod}) imply
\begin{equation}\label{eq:SphiS=phi}
\phi_{S(x)} = \phi_x.
\end{equation}
Linearizing the formulas in Proposition \ref{prop:JKV11} we obtain the fundamental connection between commutative automorphic and left Bruck formal loops. 
\begin{proposition}
\label{prop:CA-LB}
Let  $(\field[T], \dot{F})$ be a  commutative automorphic formal loop. The product
\begin{equation}\label{eq:productxy}
xy := S(r(x \1)) \dot{\backslash}(r(x\2) \cdot y)
\end{equation}
 induces a left Bruck Hopf algebra structure on $\field[T]$. We also have
 \begin{displaymath}
 x \cdot y = x\1 \phi_{x\2}(y)
 \end{displaymath}
 and 
 \begin{displaymath}
\dot{L}_a = L_a
 \end{displaymath}
for all $a \in T$. Moreover, $F(x,y) := \pi_T( xy)$ defines a  left Bruck formal loop $(\field[T],F)$ with Hopf algebra of formal distributions  $\field[F] = (\field[T],xy)$.
\end{proposition}

\subsubsection{The automorphic property of left Bruck Hopf algebras.}
The power-associa\-tivity of commutative automorphic loops implies
\begin{eqnarray}\label{eq:Hopf-pa}
 && x\1 \cdot (x\2 \cdot (\cdots (x_{(n+m-1)} \cdot x_{(n+m)}) \\
&& \quad \quad =  (x\1 \cdot  (\cdots (x_{(m-1)} \cdot x_{(m)})))\cdot  (x_{(m+1)} \cdot  (\cdots (x_{(m+n-1)} \cdot x_{(m+n)})))\nonumber
\end{eqnarray}
for all elements $x$ in any (cocommutative and coassociative) commutative automorphic Hopf algebra $\dot{H}$. With $x = a^{m+n}$, $a \in \Prim(\dot{H})$, we get $(m+n)! a^{m+n} = (m+n)! a^m \cdot a^n$, so 
\begin{displaymath}
	a^{m+n} = a^m \cdot a^n,
\end{displaymath} 
i.e., the subalgebra of $\dot{H}$ generated by any primtive element is associative and commutative. The same remains valid for left Bruck Hopf algebras, but it is false for non primitive elements. Moreover, in any (cocommutative and coassociative) left Bruck Hopf algebra $H$ we have
\begin{equation}
\label{eq:monoalternativeH}
L_{a^n} = L^n_a 
\end{equation}
for all  $a \in \Prim(H)$ and
\begin{equation}\label{eq:antipodeS}
S(x\1)(x\2 y) = \epsilon(x)y = x\1 (S(x\2) y)
\end{equation}
for all $x,y \in H$ (see \cite{MP10a}). 

The left automorphic property of left Bruck loops implies the automorphic property for left Bruck Hopf algebras:
\begin{equation}
\label{eq:lCA} l(x,y)(wz) = l(x\1,y\1)(w)l(x\2,y\2)(z)
\end{equation}
for all $x, y, w, z \in H$.

\subsubsection{The tangent Lie triple system of a left Bruck formal loop.} In dealing with the tangent space of commutative automorphic formal loops it will be very important to keep  in mind that for any primitive elements $a,b,c$  in a left Bruck Hopf algebra $H$ we have  $ab = ba$  and
\begin{equation}
\label{eq:Lts}
[[L_a,L_b],L_c] = L_{[a,b,c]}
\end{equation}
with $[a,b,c]:=a(bc)-b(ac)$  \cite{MP10a}. Beware, $H$ might fail to be commutative although primitive elements commute each other in any  left Bruck Hopf algebra.  Another important property that we will need is
\begin{equation}
\label{eq:der}
[L_a,L_b] \in \Der(H)
\end{equation}
where $\Der(H)$ stands for the Lie algebra of derivations of $H$. The triple product $[a,b,c]$ endows $\Prim(H)$ with the algebraic structure of Lie triple system.

\subsubsection{Lie triple systems and their universal enveloping algebras.} A \emph{Lie triple system} is a vector space $T$ equipped with a trilinear product $[a,b,c]$ such that for every $a,b,c,a',b' \in T$
\begin{enumerate}
\item $[a,b,c] = - [b,a,c]$,
\item $[a',b',[a,b,c]] = [[a',b',a],b,c]+[a,[a',b',b],c] + [a,b,[a',b',c]]$ and
\item (\emph{Jacobi identity}) $[a,b,c] + [b,c,a] + [c,a,b] = 0$.
\end{enumerate}
It is known \cite{PI05} that for any Lie triple system there exist a non-associative left Bruck Hopf algebra $(U(T),xy)$ and a bijective map $T \rightarrow \Prim(U(T))$ given by $a \mapsto a$ so that 
\begin{enumerate}
\item $[a,b,c] = a(bc)-b(ac)$ and
\item $(a,y,z) = -(y,a,z)$ for any $a,b,c \in T$ and $y,z \in U(T)$,
\end{enumerate}
where $(x,y,z) := (xy)z -x(yz)$ is the \emph{associator} of $x,y,z$. $U(T)$ is  called the \emph{universal enveloping algebra} of $T$ and it is universal with respect to the latter two properties:  for any unital algebra $A$ and any linear map $\iota \colon T \rightarrow A$ satisfying
\begin{displaymath}
\iota([a,b,c]) = \iota(a)(\iota(b)\iota(c)) - \iota(b)(\iota(a)\iota(c)) \quad \text{and} \quad (\iota(a),y,z) = - (y,\iota(a),z)
\end{displaymath}
for all $a,b,c \in T$ and $y, z \in A$ there exists a unique homomorphism $\varphi \colon U(T)\rightarrow A$ of unital algebras with $\varphi(a) = \iota(a)$ for any $a \in T$. In other words, $U(T)$ can be obtained as the quotient of the free unital non-associative algebra $\field\{T\}^{\#}$ on $T$ by the ideal generated by
\begin{displaymath}
R := \{ [a,b,c] - a(bc) + b(ca), (a,y,z) + (y,a,z) \mid a,b,c \in T \text{ and }  y,z \in \field\{T\}^{\#}\}.
\end{displaymath}
From this description it is clear that $U(T)$ is generated by $T$. In fact, over fields of characteristic zero
\begin{displaymath}
\{ a^n \mid a \in T, n \in \N\} \text{ spans } U(T).
\end{displaymath}
Working with these generators will greatly simplify  computations. 

The underlying coalgebra structure of $U(T)$ can be identified with that of $\field[T]$, so we may think of $U(T)$ as the Hopf algebra $\field[F]$ of formal distributions with support at the identity of some left Bruck formal loop $(\field[T], F)$ \cites{PI07,MP10}. In some sense we can say that the Lie triple system $T$ integrates to the left Bruck formal loop $(\field[T], F)$.

\subsubsection{The Poincar\'e-Birkhoff-Witt theorem.} The existence  for $U(T)$ of Poincar\'e-Birkhoff-Witt bases  is known \cite{PI05}.  Fix a basis $B$ of the vector space $T$ and a total order on $B$, then $U(T)$ has a (Poincar\'e-Birkhoff-Witt) basis of the form
\begin{displaymath}
B_{\PBW}:=\{ b_{1}(b_{2}(\cdots (b_{{l-1}} b_{l}))) \mid  b_{1}, \dots, b_{l}  \in B,  b_{1} \leq b_{2} \leq \cdots \leq b_{l}  \text{ and } l \geq 0 \}. 
\end{displaymath}
The only primitive elements in $B_{\PBW}$ are the elements with $l = 1$, i.e. the elements in $B$. We will use these bases in Section~\ref{subsec:commutative}.

\subsection{Tangent algebras of  commutative automorphic formal  loops}
Being the primitive elements of the formal loop $(\field[T],F)$, $T$ is closed under many multilinear operations expressible in terms of the product of $\field[F]$. $T$ is the \emph{tangent space} of the formal loop $F$ and it becomes a \emph{Sabinin algebra} with some of these operations \cite{MP10}. Sabinin algebras are the non-associative counterpart of Lie algebras, in the same way as formal loops are the non-associative counterpart of formal groups. The equivalence of categories $F \mapsto \field[F]$ between formal loops and connected Hopf algebras extends to an equivalence between these categories and the category of Sabinin algebras \cite{MP10}. Thus, the study of the space of primitive elements of non-associative Hopf algebras resembles the (local) Lie theory of Lie groups. However, the general notion of Sabinin algebra requires two infinite families of multilinear operations, that in our setting are expressible in terms of a ternary one. Thus, rather than using the general framework, we will focus on this ternary product.

\begin{definition}
A \emph{commutative automorphic Lie triple system} $(T, [a,b,c])$ is a Lie triple system $(T,[a,b,c])$ that satisfies
\begin{displaymath}
 [[a,b,c],a',b'] = [[a,a',b'],b,c] + [a,[b,a',b'],c] + [a,b,[c,a',b']]
\end{displaymath}
for any $a,b,c,a',b' \in T$.
\end{definition}    

Notice that the second axiom of Lie triple system is superfluous for commutative automorphic Lie triple systems.

\begin{lemma}
\label{lem:derivation}
Let $(\field[T],\dot{F})$ be a  commutative automorphic formal loop. For any $x \in \field[\dot{F}]$ and $a \in T$ we have
\begin{enumerate}
\item $\dot{l}(x,1) = \epsilon(x) \Id = \dot{l}(1,x)$ and
\item $\dot{l}(x,a)$ is a derivation of $ \field[\dot{F}]$.
\end{enumerate}
\end{lemma}
\begin{proof}
Part (1) follows from the definition of $\dot{l}$. Given $a \in T$ and  $x,w,z \in \field[\dot{F}]$
\begin{eqnarray*}
\dot{l}(x,a)(w \cdot z) &=& \dot{l}(x\1,a)(w) \cdot \dot{l}(x\2,1)(z) + \dot{l}(x\1,1)(w)  \cdot \dot{l}(x\2,a)(z) \\
&=& \dot{l}(x,a)(w) \cdot z + w \cdot \dot{l}(x,a)(z)
\end{eqnarray*}
which proves part (2).
\end{proof}

\begin{theorem}
Let $(\field[T],\dot{F})$ be a  commutative automorphic formal loop. Then $T$ is a commutative automorphic Lie triple system with the product
\begin{displaymath}
 [a,b,c] = - (a,c,b)^{\cdot}
\end{displaymath}
where $(a,b,c)^{\cdot}$ denotes the associator $(a \cdot b) \cdot c - a \cdot (b \cdot c)$.
\end{theorem}
\begin{proof}
Since $\field[F]$ is a  left Bruck Hopf algebra then $T$ is a Lie triple system with the product $[a,b,c] =  a(bc) - b(ac)$. By Proposition~\ref{prop:CA-LB} we get 
\begin{displaymath}
[[\dot{L}_a,\dot{L}_b],\dot{L}_c] = \dot{L}_{[a,b,c]}
\end{displaymath}
with 
\begin{displaymath}
[a,b,c] = a \cdot (b \cdot c) - b \cdot (a \cdot c) \stackrel{\langle 1 \rangle}{=} - (a \cdot c) \cdot b + a \cdot (c \cdot b) = - (a,c,b)^{\cdot},
\end{displaymath}
where $\langle 1 \rangle$ follows from commutativity. To show that $(T, [a,b,c])$ satisfies the commutative automorphic condition we observe that for any $a, b \in T$ and $x \in \field[\dot{F}]$
\begin{eqnarray*}
\dot{l}(a,b)(x)  &=& (a \cdot b) \dot{\backslash} x + a \cdot{\backslash} (b \cdot x) + b \dot{\backslash} (a \cdot x) + a \cdot (b \cdot x) \\
&\stackrel{\langle 1 \rangle}{=}& -(a \cdot b) \cdot x + a \cdot (b \cdot x) = - (a,b,x)^{\cdot}.
\end{eqnarray*}
where $\langle 1 \rangle$ follows from the axioms satisfied by the left division $\dot{\backslash}$ in any Hopf algebra. Hence 
\begin{displaymath}
\dot{l}(a,b)(c) = - (a,b,c)^{\cdot} = [a,c,b] = - [c,a,b].
\end{displaymath}
By Lemma \ref{lem:derivation}, we can conclude that $c \mapsto [c,a,b]$ defines a derivation of $(T,[a,b,c])$. This proves that $(T,[a,b,c])$ is a commutative automorphic Lie triple system.
\end{proof}

\begin{proposition}
\label{prop:phi}
Let $(\field[T],\dot{F})$ be a commutative automorphic formal loop, $\phi_x$ as defined in (\ref{eq:defphi}) and $xy$ as defined in (\ref{eq:productxy}). For any $a \in T$ and any $x \in \field[T]$ we have
\begin{displaymath}
 \phi_x (a) = S(x\1) (a x\2).
\end{displaymath}
\end{proposition}
\begin{proof}
The commutativity of $\field[\dot{F}]$ implies 
\begin{eqnarray*}
S(x\1)(a x\2) &=& S(x\1)(a \cdot x\2) = S(x\1)(x\2 \cdot a) = S(x\1)(x\2 \phi_{x\3}(a) ) \\
&=& \epsilon(x\1) \phi_{x\2}(a) = \phi_{x}(a).
\end{eqnarray*}
\end{proof}
  
\begin{corollary}
Two commutative automorphic formal loops  are isomorphic if and only their associated commutative automorphic Lie triple systems are isomorphic.
\end{corollary}
\begin{proof}
Formal loops are isomorphic if and only if their corresponding bialgebras of formal distributions are isomorphic as Hopf algebras \cite{MP10}. 

If the commutative automorphic formal loops $(\field[T],\dot{F})$ and $(\field[T'],\dot{F}')$ are isomorphic then the Hopf algebras $\field[\dot{F}]$ and $\field[\dot{F}']$ are isomorphic too. This isomorphism induces an isomorphism on the associated commutative automorphic Lie triples systems.

Conversely, if the commutative automorphic Lie triple systems $T$ and $T'$ associated to the commutative automorphic formal loops $(\field[T],\dot{F})$ and $(\field[T'],\dot{F}')$ are isomorphic then, by the universal property, $U(T)$ and $U(T')$ are isomorphic Hopf algebras. Since $U(T)$ (resp. $U(T')$) is isomorphic to the Hopf algebra $\field[F]$ (resp. $\field[F']$) defined in Proposition~\ref{prop:CA-LB} then $\field[F]$ is isomorphic to $\field[F']$. This isomorphism is also an isomorphism between $\field[\dot{F}]$ and $\field[\dot{F}']$.
\end{proof}

\section{Commutative automorphic Lie triple systems}
\label{sec:calts}

The property of being commutative automorphic is rather restrictive for a Lie triple system. Over algebraically closed fields of zero characteristic any simple Lie triple system contains a two-dimensional subtriple with a basis $\{e,f\}$ such that  $[e,f,e] = e$ and $[e, f, f]=-f$  \cite{Li52}. However, this subtriple does not satisfy the commutative automorphic property which, in view of the Levy decomposition for Lie triple systems, strongly determines the structure of commutative automorphic Lie triple systems.
\begin{proposition}
Any finite-dimensional commutative automorphic Lie triple system $(T,[a,b,c])$  is solvable, i.e. $T^{(n)}=0$ for some $n \geq 1$ where $T^{(1)} := [T,T,T]$ and $T^{(i+1)} :=[T,T^{(i)},T^{(i)}]$.
\end{proposition}
In this section we will describe more properties of these systems with special focus on the commutative automorphic Lie triple system $\T$ freely generated by two elements $a$ and $b$. This triple system is required to develop the Baker-Campbell-Hausdorff formula and to prove the commutativity of the commutative automorphic formal loops that we will construct to integrate commutative automorphic Lie triple systems. For any Lie triple system $(T,[a,b,c])$ let us define
\begin{eqnarray*}
R_{b,c}\colon  T & \rightarrow & T \\
a & \mapsto& [a,b,c]
\end{eqnarray*}
for any $b,c \in T$. If $T$ is a commutative automorphic Lie triple systems, these maps are derivations of $T$. The following notation will be useful:
\begin{equation*}
	[a_1,a_2,\dots,a_n] :=
	\left\{
	\begin{array}{ll}0 & \text{if } n \text{ is even} \cr
		a_1 & \text{if } n= 1\cr
		[[[a_1,a_2,a_3],\cdots],a_{n-1},a_{n}]&  \text{if } n > 1 \text{ is odd}
	\end{array}\right.
\end{equation*}
and $\stackrel{i}{\dots}c := c,c,\dots,c$  where $c$ appears  $i$ times. For instance, 
\begin{displaymath}
[a,b,\stackrel{3}{\dots}a,\stackrel{2}{\dots}b] =[a,b,a,a,a,b,b] =  [[[a,b,a],a,a],b,b].
\end{displaymath}

\begin{lemma}
\label{lem:skew}
Let $(T,[a,b,c])$ be a commutative automorphic Lie triple system and $a,b,c,a' \in T$. We have:
\begin{enumerate}
\item $[a,a',[b,a',c]]$ is skew-symmetric on $a,b$ and $c$, and
\item $[[a,a',a'],b,[c,a',a']] = 0 = [[a,a',a'],[b,a',a'],c]$.
\end{enumerate}
\end{lemma}
\begin{proof}
Since $R_{a',[c,a',b]} = [R_{a',b},R_{a',c}] = -R_{a',[b,a',c]}$ then 
\begin{equation}
\label{eq:skew1}
 [a,a',[b,a',c]] = - [a,a',[c,a',b]].
\end{equation}
By Jacobi identity we have $0 = [a,a',[b,a',c]] + [a,a',[c,a',b]]  = - [a,a',[c,b,a']] - [a,a',[a',c,b]] + [a,a',[c,a',b]]$, thus
\begin{equation}
\label{eq:skew2}
 [a,a',[c,b,a']] = 2 [a,a',[c,a',b]].
\end{equation}
With $b = a$ we get 
\begin{eqnarray*}
[a,a',[c,a,a']] &=& 2 [a,a',[c,a',a]] = - 2[a',a,[c,a',a]] = - 4 [a',a,[c,a,a']] \\
&=&  4 [a,a',[c,a,a']].
\end{eqnarray*}
Thus, $3[a,a',[c,a,a']] = 0$ and, by (\ref{eq:skew2}),
\begin{equation}
\label{eq:skew3}
[a,a',[c,a',a]] = 0.
\end{equation}
Equations (\ref{eq:skew1}) and (\ref{eq:skew3}) imply that $[a,a',[b,a',c]]$ is skew-symmetric on $a,b$ and $c$.

To prove part (2) we use Jacobi identity and part (1):
\begin{eqnarray*}
[[a,a',a'],b,[c,a',a']] &=& [a,a',[a',b,[c,a',a']]] - [a',[a,a',b],[c,a',a']]\\ 
&& - [a',b,[a,a',[c,a',a']]]  \\
&=&  0 - 0 - 0 = 0
\end{eqnarray*}
since, for instance,
 \begin{displaymath}
[a',[a,a',b],[c,a',a']] = -[[a,a',b],a',[c,a',a']]  \stackrel{\langle 1 \rangle}{=} [a',a',[c,a',[a,a',b]]] = 0 
\end{displaymath}
where $\langle 1 \rangle$ follows from part (1). In a similar way we also obtain 
\begin{displaymath}
[[a,a',a'],[b,a',a'],c] = 0.
\end{displaymath}
\end{proof}

Notice that part (2) in Lemma \ref{lem:skew} implies that $R^n_{a',a'}$ is a derivation of $(T,[a,b,c])$ for any $n \geq 1$. We can improve this result a little bit.

\begin{lemma}
Let $(T,[a,b,c])$  be a commutative automorphic Lie triple system. For any $n \geq 0$ and any $a' , b' \in T$ we have
\begin{displaymath}
R^n_{a',a'}R_{a',b'} \in \Der(T)
\end{displaymath}
where $\Der(T)$ stands for the Lie algebra of all derivations of $T$.
\end{lemma}
\begin{proof}
The result is obvious for $n = 0$, so the first case we will prove is $n = 1$, i.e. $R_{a',a'}R_{a',b'}$ is a derivation of $T$. Since $R^2_{a,a}$ is a derivation for any $a \in T$ then
\begin{equation}
\label{eq:linearizingderivations}
R_{b',a'}R_{a',a'} +  R_{a',b'}R_{a',a'} + R_{a',a'}R_{b',a'} + R_{a',a'}R_{a',b'} \in \Der(T)
\end{equation}
 for any $a',b' \in T$. However, some summands in this expression coincide. In fact, on the one hand $[R_{a',b'},R_{a',a'}] = R_{[a',a',b'],a'} + R_{a',[a',a',b']} = 0$ implies 
\begin{displaymath}
R_{a',b'}R_{a',a'} = R_{a',a'}R_{a',b'}.
\end{displaymath}
On the other hand,
\begin{eqnarray*}
R_{b',a'}R_{a',a'}(a) &=&  [[a,a',a'],b',a'] = [a,a',[a',b',a']] - [a',[a,a',b'],a'] \\
&& - [a',b',[a,a',a']]\\
&\stackrel{\langle 1 \rangle}{=}& 0 + [[a,a',b'],a',a'] - 0 = R_{a',a'}R_{a',b'}(a)
\end{eqnarray*}
and 
\begin{displaymath}
R_{a',a'}R_{b',a'} = R_{b',a'}R_{a',a'} + [R_{a',a'},R_{b',a'}] = R_{b',a'} R_{a',a'} + R_{[b',a',a'],a'},
\end{displaymath}
where $\langle 1 \rangle$ follows from part (1) in Lemma \ref{lem:skew}. Therefore, (\ref{eq:linearizingderivations}) implies 
\begin{displaymath}
4R_{a',a'}R_{a',b'}+R_{[b',a',a'],a'} \in \Der(T)
\end{displaymath}
and we get $R_{a',a'}R_{a',b'} \Der(T)$. This proves the  case $n = 1$. 

To prove the general case we can assume that $n \geq 2$. Since the maps $R^{n-1}_{a',a'}$  and $R_{a',a'}R_{a',b'}$ are derivations of $T$ then
\begin{eqnarray*}
 R_{a',a'}^{n}R_{a',b'}([a,b,c]) &=& R_{a',a'}^{n-1}\left([R_{a',a'}R_{a',b'}(a),b,c] + [a,R_{a',a'}R_{a',b'}(b),c] \right.\\ && \quad \left. + [a,b,R_{a',a'}R_{a',b'}(c)]\right)\\
&\stackrel{\langle 1 \rangle}{=}& [R_{a',a'}^{n}R_{a',b'}(a),b,c] + [a,R_{a',a'}^{n}R_{a',b'}(b),c] \\ && \quad + [a,b,R_{a',a'}^{n}R_{a',b'}(c)]
\end{eqnarray*}
where $\langle 1 \rangle$ follows from part (2) in Lemma \ref{lem:skew}. 
\end{proof}

\subsection{$2$-generated free commutative automorphic Lie triple systems}
\label{subsec:free_calts}

Inside the commutative automorphic Lie triple system $\T$ freely generated by $\{ a,b\}$ we will consider
\begin{displaymath}
\sset := \spann\langle a,b\rangle \subseteq \T.
\end{displaymath}

\begin{lemma}
\label{lem:sT}
 We have:
\begin{enumerate}
\item  $[\sset,\sset,[\T,\sset,\sset]]  = 0$ and
\item $[\T,\sset,[\sset,\sset,\sset]] = 0$.
\end{enumerate}
\end{lemma}
\begin{proof}
If  $[\sset,\sset,[\T,\sset,\sset]] \neq 0$ then $[a,b,[\T,\sset,\sset]] \neq 0$ and without loss  of generality we can assume $[a,b,[\T,b,\sset]] \neq 0$. Hence $0 \neq [\sset,b,[\T,b,a]]= [a,b,[\T,b,a]] = 0$ by Lemma \ref{lem:skew} part (1), a contradiction.  If $[\T,\sset,[\sset,\sset,\sset]] \neq 0$ then without loss of generality we can assume $[\T,a,[\sset,a,\sset]] \neq 0$ and therefore $0 \neq [\sset,a,[\T,a,\sset]] = 0$ by part (1) and Lemma~\ref{lem:skew} part (1), a contradiction.
\end{proof}

\begin{lemma}
\label{lem:symmetry}
For any $a_1,\dots, a_{n} \in \sset$ with $n \geq 0$ and any  permutation $\sigma$ of $\{1,\dots, n\}$ we have
\begin{equation}
\label{eq:symmetry}
[a,b,a_1,a_2,\dots,a_{n}] = [a,b,a_{\sigma(1)},a_{\sigma(2)},\dots,a_{\sigma(n)}].
\end{equation}
\end{lemma}
\begin{proof}
If $n$ is even then both sides of (\ref{eq:symmetry}) vanish so we can assume in the following that $n$ is odd and $n \geq 3$. We only have to prove
\begin{equation}
\label{eq:symmetry1} [a,b,a_1,\dots,a_{n-1},a_n] =  [a,b,a_1,\dots, a_n,a_{n-1}]
\end{equation}
and
\begin{equation}
\label{eq:symmetry2} [a,b,a_1,\dots,a_{n-2},a_{n-1},a_n]  = [a,b,a_1,\dots, a_{n-1},a_{n-2},a_n].
\end{equation}
On the one hand, by Jacobi identity
\begin{displaymath}
[a,b,a_1,\dots,a_{n-1},a_n] - [a,b,a_1,\dots, a_n,a_{n-1}] = - [a_{n-1},a_n,[a,b,a_1,\dots, a_{n-2}]] 
\end{displaymath}
belongs either to $[\sset,\sset,[\T,\sset,\sset]] = 0$ if $n \geq 5$ or to $[\sset,\sset,[\sset,\sset,\sset]] = 0$ if $n = 3$. This proves (\ref{eq:symmetry1}). On the other hand, to prove (\ref{eq:symmetry2}) we can assume that $n \geq 5$. For short we also set $x:= [a,b,a_1,\dots,a_{n-4}]$. Now, 
\begin{eqnarray*}
[a,b,a_1,\dots,a_{n-2},a_{n-1},a_n]  &=& [x,a_{n-3},[a_{n-2},a_{n-1},a_n]] \\
&& \quad - [a_{n-2},[x,a_{n-3},a_{n-1}],a_n] \\
&& \quad - [a_{n-2}, a_{n-1},[x,a_{n-3},a_n]]\\
&=& 0 - [a_{n-2},[x,a_{n-3},a_{n-1}],a_n] -0 \\
&=& [a,b,a_1,\dots,a_{n-1},a_{n-2},a_n]
\end{eqnarray*}
since 
\begin{displaymath}
 [x,a_{n-3},[a_{n-2},a_{n-1},a_n]] \in [\T,\sset,[\sset,\sset,\sset]] = 0
\end{displaymath}
 and 
\begin{displaymath}
[a_{n-2}, a_{n-1},[x,a_{n-3},a_n]] \in [\sset,\sset,[\T,\sset,\sset]] = 0.
\end{displaymath}
\end{proof}

\begin{lemma}
$\T$ admits a gradation $\T = \oplus_{n=1}^\infty \T_n$ where $\T_n := [\sset,\sset,\dots, \sset]$ ($\sset$ occurs $n$ times).
\end{lemma} 
\begin{proof}
	$\T$ is the quotient of the free triple system generated by $a$ and $b$--which has a grading by setting the degree of $a$ and $b$ to be $1$--by an homogeneous ideal; thus $\T$ is graded and $\T_n$ consists of elements of degree $n$. If fact, the axioms of Lie triple system allow us to write any product of $n$ elements $a_1,\dots, a_n \in \sset$ as a linear combination $\sum \alpha_\sigma [a_{\sigma(1)},\dots, a_{\sigma(n)}]$ where $\sigma$ runs on the symmetric group of degree $n$.
\end{proof}

\begin{theorem}
\label{thm:basis}
The set $B : = \{ a, b, [a,b,\stackrel{i}{\dots}a,\stackrel{j}{\dots} b] \mid i, j \geq 0, i + j \text{ odd}\, \}$ is a basis of the vector space $\T$.
\end{theorem}
\begin{proof}
Since $\T = \oplus_{n=1}^\infty \T_n$, Lemma \ref{lem:symmetry} implies that $B$ spans $\T$. The linear independence of $B$ will follow from another description of $\T$. Consider the non-unital associative algebra $\mathbb{A}$ freely generated by $\{a,b\}$ and subject to the following relations
\begin{displaymath}
aab = aba \quad \text{and} \quad bba = bab.
\end{displaymath}
$\mathbb{A}$ has a linear basis $\{a, b, aba^ib^j \mid i,j \geq 0  \}$. Moreover, $\mathbb{A}$ is a Lie algebra with the commutator product and a Lie triple system with the double commutator $[x,y,z] : =[[x,y],z]$. The subspace $T$ spanned by $\{a,b,aba^ib^j \mid i,j \geq 0, i+j \text{ odd}\,\}$ is closed under this triple product and it is a commutative automorphic Lie triple system (we leave the proof to the reader). Therefore, there exists an epimorphism of Lie triple systems $\T \rightarrow T$ induced by $a \mapsto a$ and $b\mapsto b$. Since $[a,b,\stackrel{i}{\dots}a,\stackrel{j}{\dots}b]$ is mapped to $aba^ib^j$ then $B$ has to be linearly independent, and $\T$ is isomorphic to $T$.
\end{proof}

\begin{corollary}
\label{cor:zero}
If $\{a_1,a_2,a_3\} \subset \T$ contains at least two  elements in $[\T,\T,\T]$ then $[a_1,a_2,a_3] = 0$.
\end{corollary}
\begin{proof}
We will use the isomorphism $\T \cong T$ found in the proof of Theorem~\ref{thm:basis} since it is a very attractive way of working with $\T$. The relations that define $\mathbb{A}$ are $ac = 0 = bc$ with $c:= ab - ba$. From this it is clear that $cc = c
\mathbb{A}c = 0$. Since $[T,T,T] \subseteq c \mathbb{A}$ then it is easy to conclude that $[a_1,a_2,a_3] = 0$ if two factors belong to $[T,T,T]$.
\end{proof}

\begin{corollary}
\label{cor:solvable}
We have
\begin{displaymath}
[\T,\T,[\T,\T,\T]] = 0.
\end{displaymath}
\end{corollary}
\begin{proof}
By Corollary \ref{cor:zero},  Jacobi identity and part (1) of Lemma \ref{lem:sT} we obtain $[\T,\T,[\T,\T,\T]] = [\sset,\sset,[\T,\T,\T]] \subseteq [\sset,\sset,[\T,\sset,\sset]]  = 0.$
\end{proof} 

\section{Formal integration of commutative automorphic Lie triple systems}
\label{sec:fi}
This section is the main section of the paper. Given a commutative automorphic Lie triple system $(T,[a,b,c])$ we will construct a commutative automorphic formal loop $(\field[T],\dot{F})$ so that in $\field[\dot{F}]$  
\begin{displaymath}
a \cdot( b \cdot c) - b \cdot (a \cdot c) = [a,b,c]
\end{displaymath}
holds for all $a,b,c \in T$. We will follow three natural steps:
\begin{enumerate}
\item construct  a  left Bruck formal loop $(\field[T],F)$,
\item construct the maps $\phi_x$, $x \in \field[T]$ and 
\item define $x \cdot y : = x\1 \phi_{x\2}(y)$,
\end{enumerate}
which basically amounts to retrace our steps and prove many properties we have mentioned in Section \ref{subsec:CAformalloos} until  we can claim that there exists a commutative automorphic Hopf algebra which is responsible for them. The commutative automorphic formal loop which integrates $(T,[a,b,c])$ will be the formal loop $(\field[T], \dot{F})$ with $\dot{F}(x,y) := \pi_T(x \cdot y)$. This will conclude the proof of Lie's correspondence between commutative automorphic formal loops and commutative automorphic Lie triple systems. 

\subsection{The left Bruck formal loop $U(T)$.}
As discussed in Section~\ref{subsec:CAformalloos}, any Lie triple system $(T,[a,b,c])$ has a universal enveloping algebra $U(T)$ which is a connected left Bruck Hopf algebra with $T = \Prim(U(T))$ and such that $[a,b,c] = a(bc) - b(ac)$  for every $a,b,c \in T$. As a coalgebra $U(T)$ is isomorphic to  $(\field[T], \Delta, \epsilon)$. After this identification, 
\begin{displaymath}
F(x,y) := \pi_T(xy)
\end{displaymath}
defines a left Bruck  formal loop.
\subsection{Construction of $\phi_x$.}
Let $(T,[a,b,c])$ be a commutative automorphic Lie triple system. By Proposition \ref{prop:phi} we are forced to define
\begin{eqnarray*}
\phi_x \colon T & \rightarrow & T \\
a & \mapsto & S(x\1)(a x\2)
\end{eqnarray*}
for all $x \in U(T)$. Notice that $\Delta(\phi_x(a)) = S(x\1)(a x\2) \otimes S(x\3) x\4 + S(x\1) x\2 \otimes S(x\3)(a x\4) = \phi_x(a) \otimes 1 + 1 \otimes \phi_x(a)$, i.e. $\phi_x(a) \in \Prim(U(T)) = T$ and $\phi_x$ is well defined. We can extend $\phi_x$ to 
\begin{displaymath}
\phi_x \colon \field\{ T \}^{\#} \rightarrow \field\{ T \}^{\#}
\end{displaymath}
by imposing
\begin{displaymath}
\phi_x(1) := \epsilon(x) 1 \quad \text{and} \quad \phi_x(uv) := \phi_{x\1}(u)\phi_{x\2}(v)
\end{displaymath}
for any $u, v \in \field\{T\}^{\#}$. We will prove that $\phi_x$ induces a map on the quotient $U(T)$ of $\field\{T\}^{\#}$ (see Section \ref{subsec:CAformalloos}).  However, the reader should pay attention since at this point the commutative automorphic property of $T$ is required. Observe that for any $a, a' \in T$, in $U(T)$ we will have
\begin{eqnarray*}
\phi_{a'}(a) &=& S(a')a+aa' = -a'a+aa' = 0 \text{ and }\\
\phi_{a'a'}(a) &=& (a'a')a -2a'(aa') +a(a'a') = a(a'a')-a'(aa') = [a,a',a'].
\end{eqnarray*}
Since $[a,b,c] = a(bc) - b(ac)$, then $\phi_x([a,b,c]) = [\phi_{x\1}(a),\phi_{x\2}(b),\phi_{x\3}(c)]$. With $x := a'a'$ ($a' \in T$) this equality gives 
\begin{displaymath}
[[a,b,c],a',a'] = [[a,a',a'],b,c] + [a,[b,a',a'],c] + [a,b,[c,a',a']],
\end{displaymath} i.e. $a \mapsto [a,a',a']$ is a derivation of $T$, but this is not true for general Lie triple systems. Thus, the commutative automorphic property is required.

\begin{lemma}
\label{lem:phi}
Let $(T,[a,b,c])$ be a commutative automorphic Lie triple system. For any $a' ,a \in T$ and any $m \geq 0$ we have
\begin{displaymath}
\phi_{a'^{m}}(a) = [a,\stackrel{m}{\dots}a'].
\end{displaymath}
\end{lemma}
\begin{proof}
Since $S$ is the automorphism of the algebra $U(T)$ determined by $S(a) = - a$ for any $a \in T$ \cite{MP10a}, then
\begin{eqnarray*}
S(\phi_{a'^{n}}(a)) &=& \sum_{i=0}^n \binom{n}{i} (-1)^i S(a'^i (a a'^{n-i})) = (-1)^{n+1}\sum_{i=0}^n \binom{n}{i} (-1)^i a'^i (a a'^{n-i}) \\
&=& (-1)^{n+1} \phi_{a'^n}(a),
\end{eqnarray*} 
which implies $\phi_{a'^{{2m+1}}} = 0$.  For $x := a'^{2m}$ we have
\begin{eqnarray*}
\phi_{a'(a'x)} (a)&=& a'(a'(S(x\1)))(ax\2) - 2 (a'S(x\1))(a'(ax\2)) \\
&& \quad + S(x\1)(a(a'(a' x\2)))\\
&\stackrel{\langle 1 \rangle}{=}& S(x\1) \left( a'(a'(a x\2)) - 2 a'(a(a' x\2)) + a(a'(a' x\2)) \right)\\
&=& S(x\1) \left( L_{a'}L_{a'}L_a - 2 L_{a'}L_aL_{a'} + L_{a'}L_aL_a \right)(x\2)\\
&=& S(x\1) [L_{a'},[L_{a'},L_a]](x\2) \stackrel{\langle 2 \rangle}{=} S(x\1)L_{[a,a',a']}(x\2) \\
&=& \phi_{x}([a,a',a'])
\end{eqnarray*}
where $\langle 1 \rangle$ follows from  (\ref{eq:monoalternativeH})  and  $\langle 2 \rangle$ follows from (\ref{eq:Lts}). 
\end{proof}

\begin{lemma}
\label{lem:autphi}
Let $T$ be a commutative automorphic Lie triple system. For any $x \in U(T)$ and any $a,b,c$ in $T$ we have 
\begin{displaymath}
\phi_x([a,b,c]) = [\phi_{x\1}(a),\phi_{x\2}(b),\phi_{x\3}(c)].
\end{displaymath}
\end{lemma}
\begin{proof}
Without loss of generality we can assume that $x = a'^n$ for some $a' \in T$ and $n \geq 0$.  The case $n = 0$ is trivial so we also assume that $n \geq 1$. If $n$ is odd then $x\1, x\2$ or $x\3$ involves an odd power of $a'$ so both sides of the formula in the statement vanish. If $n$ is even, $n = 2m$ for some $m$, then by Lemma~\ref{lem:phi} and  Lemma~\ref{lem:skew} part (2) $\phi_x = R_{a',a'}^m$ is a derivation of $T$. Thus,
\begin{displaymath}
\phi_x([a,b,c]) = [\phi_{x}(a),b,c] + [a,\phi_x(b),c] + [a,b,\phi_x(c)]. 
\end{displaymath}
Again, Lemma \ref{lem:phi} and Lemma \ref{lem:skew}  part (2)  ensure that this equality is equivalent to the equality in the statement.
\end{proof}

\begin{proposition}
\label{prop:autphi}
Let $T$ be a commutative automorphic Lie triple system. There exist linear maps  $\phi_x \colon U(T) \rightarrow U(T)$ ($x\in U(T)$) such that $\phi_x (a)=S(x\1)(a x\2)$ for any $a \in T$ and 
\begin{equation}
\label{eq:fa}
\phi_x(yz) = \phi_{x\1}(y) \phi_{x\2}(z)
\end{equation}
for any $y,z \in U(T)$.
\end{proposition}
\begin{proof}
Up to isomorphism, $U(T)$ is the quotient algebra of the unital free algebra $\field\{ T \}^\#$ by the ideal generated by
\begin{displaymath}
R := \{ [a,b,c] - a(bc) + b(ca), (a,y,z) + (y,a,z) \mid a,b,c \in T \text{ and }  y,z \in \field\{T\}^{\#}\}.
\end{displaymath}
We extend $\phi_x \colon T \rightarrow T$ to $\field\{ T \}^\#$ by imposing
\begin{displaymath}
\phi_x(1) := \epsilon(x) 1 \quad \text{and} \quad \phi_x(uv) := \phi_{x\1}(u)\phi_{x\2}(v)
\end{displaymath}
for any $u, v \in \field\{T\}^{\#}$. By Lemma \ref{lem:autphi}, $\phi_x(R)$ consists of linear combinations of elements in $R$, so $\phi_x$  preserves the ideal generated by $R$. This proves that $\phi_x$ induces a map $\phi_x \colon U(T) \rightarrow U(T)$ that fulfills all our requirements.
\end{proof}

\subsection{The non-associative Hopf algebra $\dot{U}(T)$.}
We can define a new product on the vector space $U(T)$ by
\begin{equation}
\label{eq:CAUEA}
x \cdot y := x\1 \phi_{x\2}(y).
\end{equation}
This product has the same unit element as the product $xy$ since $1 \cdot y = 1 \phi_1(y) = y = y\1 \epsilon(y \2) = y \cdot 1$. To avoid confusions, the algebraic structure
\begin{displaymath}
(U(T), \Delta, \epsilon, x \cdot y, 1)
\end{displaymath}
will be denoted by $\dot{U}(T)$, i.e. $\dot{U}(T)$ is the same vector space as $U(T)$ endowed with the same coalgebra structure and the same unit element but with a different product. Our goal is to prove that $\dot{U}(T)$ is a commutative automorphic Hopf algebra, $T = \Prim(\dot{U}(T))$  and $[a,b,c] = - (a,c,b)^{\cdot}$ for any $a,b,c \in T$. In fact, we already know that $T = \Prim(\dot{U}(T))$ since the coalgebra structure of $\dot{U}(T)$ is the same as the coalgebra structure of $U(T)$.

\begin{lemma}
\label{lem:L}
Let $(T,[a,b,c])$ be a commutative automorphic Lie triple system. For any $a \in T$ and any $y \in \dot{U}(T)$ we have
\begin{enumerate}
\item $\dot{L}_a = L_a$ and
\item $a \cdot y = y \cdot a$.
\end{enumerate}
\end{lemma}
\begin{proof}
By definition of $x \cdot y$ and $\phi_x$, 
\begin{eqnarray*}
	a \cdot y &=& a\1 \phi_{a\2}(y) = a y + \phi_a(y) \\
	&\stackrel{\langle 1 \rangle}{=} & ay \stackrel{\langle 2 \rangle}{=} y\1 (S(y\2)(a y\3)) = y\1 \phi_{y\2}(a) \\
	&=& y \cdot a.
\end{eqnarray*} 
where $\langle 1 \rangle$ follows from Proposition~\ref{prop:autphi} and Lemma~\ref{lem:phi},  and $\langle 2 \rangle$ follows from (\ref{eq:antipodeS}). This proves the lemma. 
\end{proof}

\begin{proposition}
Let $(T,[a,b,c])$ be a commutative automorphic Lie triple system. For any $a,b,c \in T$ we have 
\begin{displaymath}
[a,b,c] = - (a,c,b)^{\cdot},
\end{displaymath}
where $(x,y,z)^{\cdot}$ stands for the associator $(x\cdot y) \cdot z - x \cdot (y \cdot z)$ of $x,y,z$ in $\dot{U}(T)$.
\end{proposition}
\begin{proof}
By Lemma \ref{lem:L}, $-(a,c,b)^{\cdot} = -(a \cdot c) \cdot b + a \cdot (c \cdot b) = a \cdot (b \cdot c) - b \cdot (a \cdot c) = a(bc) - b(ac) = [a,b,c]$.
\end{proof}

\begin{lemma}
\label{lem:deltaphi}
Let $(T,[a,b,c])$ be a commutative automorphic Lie triple system. For any $x,y \in U(T)$ we have $\Delta(\phi_x(y)) = \phi_{x\1}(y\1) \otimes \phi_{x\2}(y\2)$ and $\epsilon(\phi_x(y)) = \epsilon(x) \epsilon(y)$.
\end{lemma}
\begin{proof}
For $y=1$ we have $\Delta(\phi_x(1)) = \Delta(\epsilon(x)1) = \epsilon(x) 1 \otimes 1 = \epsilon(x\1) 1 \otimes \epsilon(x\2)1 = \phi_{x\1}(1) \otimes \phi_{x\2}(1)$. For $y=a \in T$ we get $\Delta(\phi_x(a)) = \Delta (S(x\1)(a x\2)) = S(x\1)(a x\2) \otimes  S(x\3) x\4 + S(x\1) x\2 \otimes  S(x\3) (a x\4) = S(x\1)(a x\2) \otimes  \epsilon(x\3)1 + \epsilon(x\1) 1\otimes  S(x\2) (a x\3) = \phi_{x\1}(a) \otimes \phi_{x\2}(1) + \phi_{x\1}(1) \otimes \phi_{x\2}(a) = \phi_{x\1}(a\1) \otimes \phi_{x\2}(a\2)$. Since $T$ generates the algebra $U(T)$ these initial steps and (\ref{eq:fa}) show that the result is true. Notice that we have freely used (\ref{eq:antipodeS}) and $\Delta S = (S \otimes S) \Delta$, which is true since both maps are homomorphisms $U(T) \rightarrow U(T) \otimes U(T)$ of unital algebras that agree on the generator set $T$.
\end{proof}

\begin{proposition}
Let $(T,[a,b,c])$ be a commutative automorphic Lie triple system. Then $\dot{U}(T)$ is a connected Hopf algebra.
\end{proposition}
\begin{proof}
$U(T)$ is a connected Hopf so we only have to check that $\Delta$ and $\epsilon$ are homomorphisms of unital algebras, i.e. $\Delta(x \cdot y) = x\1 \cdot y \1 \otimes x\2 \cdot y\2$, $\Delta(1) = 1 \otimes 1$, $\epsilon(x \cdot y) = \epsilon(x) \epsilon(y)$ and $\epsilon(1) = 1$, which is an easy consequence of Lemma \ref{lem:deltaphi}.
\end{proof}

\subsection{$\dot{U}(T)$ satisfies the left automorphic property}
We would like to prove
\begin{displaymath}
\dot{l}(x,y)(w \cdot z) = \dot{l}(x\1,y\1)(w) \cdot \dot{l}(x\2,y\2)(z)
\end{displaymath}
for any $x,y,w,z \in \dot{U}(T)$ (recall Definition \ref{def:CA}). The proof of this result is quite straightforward with no interesting ideas coming into play so the reader is advised to skip this part at first reading and come back to it later.

\begin{lemma}
\label{lem:SlS}
Let $(T,[a,b,c])$ be a commutative automorphic Lie triple system. We have
\begin{displaymath}
S \phi_x S = \phi_{S(x)} = \phi_x \quad \text{and} \quad Sl(x,y)S = l(x,y)
\end{displaymath}
for any $x,y \in U(T)$.
\end{lemma}
\begin{proof} 
Evaluating at $1$ we get
\begin{displaymath}
S \phi_x S(1)= S\phi_x(1) = \epsilon(x) 1= \left\{ \begin{array}{l}  \phi_x(1), \\ \ \\
\epsilon(S(x)) 1 = \phi_{S(x)}(1). \end{array}   \right. 
\end{displaymath}
Evaluating at $a \in T$ we get
$S\phi_x S(a) = S\phi_x(-a) = \phi_x(a)$ since $\phi_x(a) \in T$. We also have $\phi_{S(x)}(a) = x\1(aS(x\2)) = S^2(x\1)(S^2(a)S(x\2)) =S(S(x\1)(S(a) x\2)) = S\phi_xS(a)$. Since $T$ generates $U(T)$, the general case follows from (\ref{eq:fa}) and the fact that $S$ is an automorphism of order $2$. The second identity follows in a similar way because $l(x,y)(T) \subseteq T$.
\end{proof}

\begin{lemma}\label{lem:flipphi}
Let $(T,[a,b,c])$ be a commutative automorphic Lie triple system. We have
\begin{displaymath}
\phi_z \phi_x = \phi_{\phi_{z\1}(x)}\phi_{z\2}
\end{displaymath}
for all $x,z \in U(T)$.
\end{lemma}
\begin{proof}
Clearly, the identity in the statement holds when both sides are evaluated at $1$. For any $ a \in T$
\begin{eqnarray*}
\phi_z\phi_x(a) &=& \phi_z(S(x\1)(a x\2)) = \phi_{z\1}(S(x\1)) (\phi_{z\2}(a) \phi_{z\3}(x\2))\\
&=& S(\phi_{z\1}(x\1)) (\phi_{z\2}(a) \phi_{z\3}(x\2)) = \phi_{\phi_{z\1}(x)}\phi_{z\2}(a).
\end{eqnarray*}
Having proved that both maps coincide on the generators of $U(T)$, the result follows from (\ref{eq:fa}).
\end{proof}

\begin{lemma}\label{lem:phi_aut_dot}
Let $(T,[a,b,c])$ be a commutative automorphic Lie triple system. We have
\begin{displaymath}
\phi_x(y \cdot z) = \phi_{x\1}(y) \cdot \phi_{x\2}(z)
\end{displaymath}
for any $x,y,z \in \dot{U}(T)$.
\end{lemma}
\begin{proof}
By the definition of the product of $\dot{U}(T)$ and Lemma~\ref{lem:flipphi} we have
\begin{eqnarray*}
\phi_x(y \cdot z) &=& \phi_x(y\1 \phi_{y\2}(z)) = \phi_{x\1}(y\1) \phi_{x\2}\phi_{y\2}(z) \\&=& \phi_{x\1}(y\1) \phi_{\phi_{x\2}(y\2)}\phi_{x\3}(z) =\phi_{x\1}(y) \cdot \phi_{x\2}(z)
\end{eqnarray*}
\end{proof}

\begin{lemma}
\label{lem:laut}
Let $(T,[a,b,c])$ be a commutative automorphic Lie triple system. For every $x,y,z \in U(T)$ we have 
\begin{displaymath}
 l(x,y) \phi_z = \phi_{l(x\1,y\1)(z) }l(x\2,y\2).
\end{displaymath}
\end{lemma}
\begin{proof}
The equality holds when both sides are evaluated at $1$. Given $a \in T$, 
\begin{eqnarray*}
l(x,y) \phi_z(a) &=& l(x,y)(S(z\1(a z\2)))
\\
&\stackrel{\langle 1 \rangle}{=}& l(x\1,y\1)\left(S(z\1)\right)\left(l(x\2,y\2)(a)l(x\3,y\3)(z\2)\right)\\
&\stackrel{\langle 2 \rangle}{=}& S\left(l(x\1,y\1)(z\1)\right)\left(l(x\2,y\2)(a)l(x\3,y\3)(z\2)\right)\\
&=& \phi_{l(x\1,y\1)(z)}l(x\2,y\2)(a)
\end{eqnarray*}
where $\langle 1 \rangle$ follows from (\ref{eq:lCA}) and $\langle 2 \rangle$ follows from Lemma \ref{lem:SlS}. Since $T$ generates $U(T)$ as a unital algebra, we can conclude the proof by (\ref{eq:fa}) and (\ref{eq:lCA}).
\end{proof}

\begin{lemma} \label{lem:l_aut_dot}
Let $(T,[a,b,c])$ be a commutative automorphic Lie triple system. For every $x,y,w, z \in \dot{U}(T)$ we have
\begin{displaymath}
l(x,y)(w \cdot z) = l(x\1,y\1)(w) \cdot l(x\2,y\2)(z).
\end{displaymath}
\end{lemma}
\begin{proof}
It is enough to observe that
\begin{eqnarray*}
l(x,y)(w\cdot z) &=& l(x,y)(w\1 \phi_{w\2}(z)) \\
&\stackrel{\langle 1 \rangle}{=}& l(x\1,y\1)(w\1) l(x\2,y\2)(\phi_{w\2}(z))\\
&\stackrel{\langle 2 \rangle}{=}& l(x\1,y\1)(w\1)\phi_{l(x\2,y\2)(w\2)}(l(x\3,y\3)(z))\\
&=& l(x\1,y\1)(w) \cdot l(x\2,y\2)(z)
\end{eqnarray*}
where $\langle 1 \rangle$ follows from (\ref{eq:lCA}) and  $\langle 2 \rangle$ follows from Lemma \ref{lem:laut}.
\end{proof}

The set $\Hom_{\field}(U(T), \Endo_{\field}(U(T)))$ is an associative algebra with the convolution product
\begin{displaymath}
(f*g)_x= f_{x\1} g_{x\2}
\end{displaymath}
where $f_x \in \Endo_{\field}(U(T))$ stands for the image of $x \in U(T)$ under the map $f \in \Hom_{\field}(U(T) ,\Endo_{\field}(U(T)))$. The unit element of this associative algebra is the map $x \mapsto \epsilon(x) \Id$. It is easy to prove the existence of left and right inverses (that must coincide) of $\phi\colon x \mapsto \phi_x$ in $\Hom_{\field}(U(T),\Endo_{\field}(U(T)))$. This inverse will be denote by $\phi'$. The existence of $\phi'$ can be obtained  inductively by $\phi'_1 := \Id$ and $\phi'_{x\1 a} \phi_{x\2} + \phi'_{x\1}\phi_{x\2 a} = 0$ for any $x \in U(T)$ and $a \in T$, which gives the formula for $\phi'_{xa}$ in terms of previously defined maps. Thus
\begin{displaymath}
\phi'_{x\1}\phi_{x\2} (y)= \epsilon(x) y = \phi_{x\1}\phi'_{x\2}(y)
\end{displaymath}
for any $x,y \in U(T)$.

\begin{lemma}\label{lem:phi_prime_aut_dot}
Let $(T,[a,b,c])$ be a commutative automorphic Lie triple system. For every $x,y,z \in U(T)$ we have
\begin{enumerate}
\item $\phi'_{x}\phi'_{y} = \phi'_{\phi'_{x\1}(y)}\phi'_{x\2}$,
\item $\phi'_x(yz) = \phi'_{x\1}(y) \phi'_{x\2}(z)$ and
\item $\phi'_x(y \cdot z) = \phi'_{x\1}(y) \cdot \phi'_{x\2}(z)$.
\end{enumerate}
\end{lemma}
\begin{proof}
To prove the first equality we observe that 
\begin{eqnarray*}
\phi'_{x\1}(y)\phi'_{x\2}(z) &=& \phi'_{x\1}(\phi_{x\2}(\phi'_{x\3}(y)\phi'_{x\4}(z))) \\
&=& \phi'_{x\1}(\phi_{x\2}\phi'_{x\3}(y)\phi_{x\4}\phi'_{x_{(5)}}(z)) = \phi'_{x}(y z).
\end{eqnarray*} 
The other identities follow in a similar way.
\end{proof}

\begin{lemma}
\label{lem:formulaldot}
Let $(T,[a,b,c])$ be a commutative automorphic Lie triple system. For every $x,y \in U(T)$ we have
\begin{displaymath}
\dot{l}(x,y) = \phi'_{x\1 \cdot y\1} l(x\2, \phi_{x\3}(y\2))\phi_{x\4} \phi_{y\3}.
\end{displaymath}
\end{lemma}
\begin{proof}
\begin{eqnarray*}
&&(x\1 \cdot y\1) \cdot \phi'_{x\2 \cdot y\2}\left( l(x\3, \phi_{x\4}(y\3)) (\phi_{x_{(5)}}(\phi_{y\4}(z))) \right) \\
&& \quad = (x\1 \cdot y\1)  l(x\2, \phi_{x\3}(y\2)) (\phi_{x\4}(\phi_{y\3}(z)))\\
&& \quad = x\1 (\phi_{x\2}(y\1)(\phi_{x\3}(\phi_{y\2}(z)))) = x\1 \phi_{x\2}(y\1 \phi_{y\2}(z)) \\
&& \quad = x \cdot (y \cdot z).
\end{eqnarray*}
Dividing on the left we get 
\begin{displaymath}
 \phi'_{x\1 \cdot y\1} l(x\2, \phi_{x\3}(y\2))\phi_{x\4} \phi_{y\3}(z) = (x\1 \cdot y\1) \dot{\backslash} (x\2 \cdot (y\2 \cdot z)) = \dot{l}(x,y)(z).
\end{displaymath}
\end{proof}

The description of $\dot{l}(x,y)$ in Lemma~\ref{lem:formulaldot} and the properties of $\phi_x$, $l(x,y)$ and $\phi'_x$ proved in lemmas \ref{lem:phi_aut_dot}, \ref{lem:l_aut_dot} and \ref{lem:phi_prime_aut_dot} lead to the desired commutative automorphic property of $\dot{U}(T)$.
\begin{proposition}
\label{prop:automorphicHopf}
Let $(T,[a,b,c])$ be a commutative automorphic Lie triple system. For every $x,y,w,z \in U(T)$ we have
\begin{eqnarray*}
\dot{l}(x,y)(wz) &=& \dot{l}(x\1,y\1)(w) \dot{l}(x\2,y\2)(z) \quad\text{and} \\
\dot{l}(x,y)(w \cdot z) &=& \dot{l}(x\1,y\1)(w) \cdot \dot{l}(x\2,y\2)(z).
\end{eqnarray*}
\end{proposition}
\subsection{$\dot{U}(T)$ is commutative}
\label{subsec:commutative}
To prove that $\dot{U}(T)$ is commutative, i.e. $x \cdot y = y \cdot x$ for any $x, y \in \dot{U}(T)$ there is no loss of generality in assuming that $x = a^m$ and $y = b^n$ for some $a, b \in T$. This observation is crucial since it ensures that we can restrict our study to commutative automorphic Lie triple systems generated by two elements. In fact  it is enough to prove the result for the commutative automorphic Lie triple system $\T$ freely generated by $\{a,b\}$ (see Section~\ref{subsec:free_calts}) since the epimorphism $U(\T) \rightarrow U(T)$ provided by the universal property of $U(\T)$ induces an epimorphism $\dot{U}(\T) \rightarrow \dot{U}(T)$. To prove the commutativity of $\dot{U}(\T)$ we will compare some terms in the expansions of $a^m \cdot b^n$ and $b^n \cdot a^m$ as linear combinations of a Poincar\'e-Birkhoff-Witt basis of $U(\T)$. 

First we will establish some preliminary results. Fix the linear basis $B : = \{ a, b, [a,b,\stackrel{i}{\dots}a,\stackrel{j}{\dots} b] \mid i, j \geq 0, i + j \text{ odd}\, \}$ of $\T$ and the deg-lex order  with $a < b$  for this basis. Thus, $a < b <  [a,b,\stackrel{i}{\dots}a,\stackrel{j}{\dots} b]$ and $U(\T)$ has a Poincar\'e-Birkhoff-Witt linear basis of the form
\begin{displaymath}
B_{\PBW}:=\{ b_{1}(b_{2}(\cdots (b_{{l-1}} b_{l}))) \mid  b_{1}, \dots, b_{l}  \in B,  b_{1} \leq b_{2} \leq \cdots \leq b_{l}  \text{ and } l \geq 0 \}. 
\end{displaymath}
\begin{definition}
The \emph{length} $l(x)$ of $x:= b_{1}(b_{2}(\cdots (b_{{l-1}} b_{l}))) \in B_{\PBW}$ is $l$. The vector space spanned by all the elements in $B_{\PBW}$ of length $l$ will be denoted by $U(\T)_l$ (or $\dot{U}(\T)_l$). Clearly $U(\T)_1 = \T$. The \emph{primitive component} of $y \in U(\T)$ is the element $y_1$ in the expansion $y = y_0 + y_1 + y_2 + \cdots \in U(\T) = \oplus_{i=0}^\infty U(\T)_i$ where $y_i \in U(\T)_i$ for all $i$. To indicate that the primitive components of $x$ and $y$ are the same we will use the  notation $x \equiv y$.
\end{definition}

Our next lemma shows that the order in $ B \cap [\T,\T,\T]$ is irrelevant.
\begin{lemma}
Let $b_1(b_2(\cdots (b_{l-1}b_l)))$ be an element in $B_{\PBW}$. If $b_{k} \in  B \cap [\T,\T,\T]$ then  for any permutation $\sigma$ of $\{k,k+1,\dots, l\}$ we have
\begin{displaymath}
b_1(b_2(\cdots (b_{l-1}b_l))) = b_1(b_2(\cdots (b_{\sigma(k)}(\cdots(b_{\sigma(l-1)}b_{\sigma(l)})))))).
\end{displaymath}
\end{lemma}
\begin{proof}
Since $b_{k} \in B \cap [\T,\T,\T]$ then $b_{k+1},\dots, b_l \in B \cap [\T,\T,\T]$ and  $b_1(b_2(\cdots (b_{l-1}b_l)))$ equals
\begin{eqnarray*}
 && b_1(\cdots( b_{k+1}(b_{k}(b_{k+2}\cdots(b_{l-1}b_l))))) \\  
&& \quad +\sum_{i=k+2}^l b_1(\cdots (b_{k-1}(b_{k+2}(\cdots( [b_{k},b_{k+1},b_i](\cdots (b_{l-1}b_l)))))))
\end{eqnarray*}
where we have used (\ref{eq:Lts}) and (\ref{eq:der}).  By Corollary \ref{cor:solvable} we get
\begin{displaymath}
  b_1(\cdots( b_{k}(b_{k+1}(b_{k+2}\cdots(b_{l-1}b_l))))) = b_1(\cdots( b_{k+1}(b_{k}(b_{k+2}\cdots(b_{l-1}b_l)))))
\end{displaymath}
if $b_k \in B \cap [\T,\T,\T]$. Since $b_{k+1},\dots, b_l $ also belong to $B \cap [\T,\T,\T]$ the result follows.
\end{proof}

The basic tool to compute primitive components is the next lemma.
\begin{lemma}
\label{lem:primitivecomponentzero}
Let $a_1,\dots, a_l $ be elements in $ \T$. If $a_k \in [\T,\T,\T]$  for some $k >1$ then the primitive component of $a_1(a_2(\cdots(a_{l-1}a_l)))$ is zero.
\end{lemma}
\begin{proof}
Without loss of generality we can assume that $a_1,\dots, a_l \in B$. What we cannot assume is that $a_1(a_2(\cdots(a_{l-1}a_l)))$ belongs to the basis $B_{\PBW}$. To reorder the factors we  use 
\begin{equation}
\label{eq:flip}
L_{a_i}L_{a_i+1} = L_{a_{i+1}}L_{a_i} + [L_{a_i},L_{a_{i+1}}] \quad \text{with}\quad [L_{a_i},L_{a_{i+1}}] \in \Der(U(\T)).
\end{equation}
Thus, $a_1(a_2(\cdots(a_{l-1}a_l)))$ can be written as $a_1(a_2(\cdots(a_{k+1}(a_k(\cdots (a_{l-1}a_l))))))$ plus a linear combination of products $a'_1(a'_2(\cdots(a'_{l-3}a'_{l-2})))$ with $a'_1,\dots, a'_{l-2} \in\T$ where  $a'_i \in [\T,\T,\T]$ for some $i \geq k$.  The result will immediately  follow by induction on $l-k$ once we prove the initial case $l-k = 0$. If $k = l$ then, by Corollary \ref{cor:solvable}, $[\T,\T,a_l] = 0$ and all the nonzero products $a'_1(a'_2(\cdots(a'_{l-3}a'_{l-2})))$ contain at least two factors (and  $a'_{l-2} = a_l$). Therefore, after applying (\ref{eq:flip}) several times we end up with the expansion of $a_1(a_2(\cdots(a_{l-1}a_l)))$ as a linear combination of elements in  $B_{\PBW}$ of length $\geq 2$. Thus, the primitive component of  $a_1(a_2(\cdots(a_{l-1}a_l)))$ is zero.
\end{proof}

\begin{lemma}
\label{lem:primitivecomponent}
Let $a_1,\dots, a_l $ be elements in $ \T$. If $a_1 \in [\T,\T,\T]$ then the primitive component of $a_1(a_2(\cdots (a_{l-1}a_{l})))$ is $[a_1,a_2,\dots,a_l]$.
\end{lemma}
\begin{proof}
The cases $l = 1$ and $l = 2$ are trivial, so let us assume that $l \geq 3$. We have 
\begin{eqnarray*}
a_1(a_2(\cdots (a_{l-1}a_{l}))) &=& a_2(a_1(\cdots (a_{l-1}a_{l}))) + \sum_{i=3}^l a_3(\cdots ([a_1,a_2,a_i](\cdots(a_{l-1}a_l)))) \\
&\equiv&  [a_1,a_2,a_3](a_4(\cdots (a_{l-1}a_{l}))) \equiv \cdots \\
&\equiv& [a_1,a_2,\dots, a_l].
\end{eqnarray*}
\end{proof}

Now that we have a method to compute primitive components we can prove the commutativity of $\dot{U}(T)$.

\begin{proposition}\label{prop:main_proposition}
For any commutative automorphic Lie triple system $(T,[a,b,c])$ the algebra $\dot{U}(T)$ is commutative.
\end{proposition}
\begin{proof}
We will prove by double induction on $m$ and $n$  that 
\begin{displaymath}
a^m \cdot b^n = b^n \cdot a^m
\end{displaymath}
holds in $\dot{U}(\T)$  for any $m,n \geq 0$, being the case $m = 0$ trivial. The case $m = 1$ follows from
\begin{displaymath}
 b^n \cdot a = b^n\1 \phi_{b^n\2}(a) = b^n\1 (S(b^n\2(a b^n\3))) = ab^n= a \cdot b^n.
\end{displaymath}
So, we fix 
\begin{displaymath}
m \geq 2
\end{displaymath}
and we assume 
\begin{equation}
\label{eq:commutation1}
a^i \cdot b^n = b^n \cdot a^i \quad \text{for all } n \geq 0 \text{ and } 0 \leq i \leq m-1.
\end{equation}
Now we fix 
\begin{displaymath}
n \geq 2
\end{displaymath}
and we assume 
\begin{equation}
\label{eq:commutation2}
a^m \cdot b^j = b^j \cdot a^m  \quad \text{for all }  0 \leq j \leq n-1.
\end{equation}
We will prove in the next two lemmas that $a^m \cdot b^n = b^n \cdot a ^m$  for these fixed $m$ and $n$. By induction on $n$ this shows that $a^m \cdot b^n = b^n \cdot a^m$ for all $n \geq 0$, and by induction on $m$ we  get that this equality  holds for all $n, m \geq 0$.
\begin{lemma}
\label{lem:tau}
In $\dot{U}(\T)$ we have
\begin{displaymath}
a^m \cdot b^n - b^n \cdot a^m  \in \T.
\end{displaymath}
\end{lemma}
\begin{proof}
\begin{eqnarray*}
\Delta(a^m \cdot b^n - b^n \cdot a^m )  &=& a^m\1 \cdot b^n\1 \otimes a^m\2 \cdot b^n\2 - b^n\1 \cdot a^m\1 \otimes b^n\2 \cdot a^m\2 \\
&=& (a^m \cdot b^n- b^n \cdot a^m )\otimes 1 + 1 \otimes (a^m \cdot b^n - b^n \cdot a^m )
\end{eqnarray*}
by (\ref{eq:commutation1}) and (\ref{eq:commutation2}).
\end{proof}

\begin{lemma}
In $\dot{U}(\T)$ we have
\begin{displaymath}
a^m \cdot b^n = b^n \cdot a^m.
\end{displaymath}
\end{lemma}
\begin{proof}
Our strategy is to expand $a^m \cdot b^n$ and $b^n \cdot a^m$ in terms of $B_{\PBW}$ to compare the primitive components of them. This  amounts to discard all the basic elements not belonging to $B$ that appear in these expansions. Notice that if the primitive components agree then $a^m\cdot b^n = b^n \cdot a^m$ because $a^m \cdot b^n - b^n \cdot a^m \in \T$. 

On the one hand
\begin{eqnarray*}
a^m \cdot b^n &=& a^m\1 \phi_{a^m\2}(b^n) \stackrel{\langle 1 \rangle }{\equiv}\phi_{a^m\1}(b) (\cdots (\phi_{a^m_{(n-1)}}(b)\phi_{a^m_{(n)}}(b)))\\
& \stackrel{\langle 2 \rangle}{\equiv}& \phi_{a^m}(b) (b( \cdots (bb))) \stackrel{\langle 3 \rangle}{\equiv} [[b,a,\stackrel{m-1}{\dots}a],\stackrel{n-1}{\dots}b]
\end{eqnarray*}
where $\langle 1 \rangle$ follows from the fact that the elements in $a^i U(\T)_j$ have zero primitive component if $i , j\geq 1$ ($a$ is the lowest element in $B$) and congruences $\langle 2 \rangle$ and $\langle 3 \rangle$ are consequences of Lemma \ref{lem:primitivecomponentzero} and Lemma \ref{lem:primitivecomponent} respectively. Beware, the element we have obtained is not $ [b,a,\stackrel{m-1}{\dots}a,\stackrel{n-1}{\dots}b] $.  On the other hand,
\begin{eqnarray*}
b^n a^m &=&  b(\cdots(b(a(\cdots(aa))))) \stackrel{\langle 1 \rangle}{\equiv}b(\cdots(a(b(\cdots(aa))))) \equiv \cdots \\
& \equiv& a(b(\cdots(b(a(\cdots(aa)))))) + [b,a,b](b(\cdots(b(a(\cdots(aa)))))) \\
& \stackrel{\langle 2 \rangle}{\equiv}& [b,a,b](b(\cdots(b(a(\cdots(aa))))))  \stackrel{\langle 3 \rangle}{\equiv}  [b,a,\stackrel{n-1}{\dots}b,\stackrel{m-1}{\dots}a] 
\end{eqnarray*}
where $\langle 1 \rangle$ follows from Lemma \ref{lem:primitivecomponentzero} ($n \geq 2$ is required), $\langle 2 \rangle$  follows again from the fact that the elements in $a^i U(\T)_j$ have zero primitive component if $i , j\geq 1$ and $\langle 3 \rangle$ is a consequence of Lemma \ref{lem:primitivecomponent}. Therefore,
\begin{eqnarray*}
b^n \cdot a^m &=& b^n\1 \phi_{b^n\2}(a^m) = b^n\1(\phi_{b^n\2}(a)( \cdots (\phi_{b^n_{(m-1)}}(a) \phi_{b^n_{(m)}}(a) )))\\
&\stackrel{\langle 1 \rangle}{\equiv}& b^n a^m + \phi_{b^n}(a)a^{m-1}\\
&\stackrel{\langle 2 \rangle}{\equiv}& [b,a,\stackrel{n-1}{\dots}b,\stackrel{m-1}{\dots}a] +[ [a,b,\stackrel{n-1}{\dots}b],\stackrel{m-1}{\dots}a]
\end{eqnarray*}
where $\langle 1 \rangle$ and  $\langle 2 \rangle$ follow again from  Lemma \ref{lem:primitivecomponentzero} and Lemma \ref{lem:primitivecomponent} respectively. Finally, we will check 
\begin{displaymath}
 [[b,a,\stackrel{m-1}{\dots}a],\stackrel{n-1}{\dots}b] = [b,a,\stackrel{n-1}{\dots}b,\stackrel{m-1}{\dots}a] +[ [a,b,\stackrel{n-1}{\dots}b],\stackrel{m-1}{\dots}a].
\end{displaymath}
In case that $m$ and $n$ have the same parity both sides of the equality vanish. In case that $m$ is odd and $n$ is even the equality is  $0 =  [b,a,\stackrel{n-1}{\dots}b,\stackrel{m-1}{\dots}a] +[a,b,\stackrel{n-1}{\dots}b,\stackrel{m-1}{\dots}a]$, which is true. If $m$ is even and $n$ is odd then the equality is $ [b,a,\stackrel{m-1}{\dots}a,\stackrel{n-1}{\dots}b] = [b,a,\stackrel{n-1}{\dots}b,\stackrel{m-1}{\dots}a]$, which is true again by Lemma \ref{lem:symmetry}.
\end{proof}
This concludes the proof of Proposition~\ref{prop:main_proposition}.
\end{proof}

\section{A Baker-Campbell-Hausdorff formula for commutative automorphic loops}
\label{sec:bch}

In this section we will compute a Baker-Campbell-Hausdorff formula for commutative automorphic formal loops. Let $\T$ be the commutative automorphic Lie triple system freely generated by $\{ a, b\}$,  $U(\T)$ its left Bruck universal enveloping algebra and $\overline{U(\T)}$  its completion with respect to the $I$-adic topology ($I = \ker \epsilon$). Endowed with the continuous extensions of the operations of $U(\T)$, $\overline{U(\T)}$ is a topological left Bruck Hopf algebra (which basically amounts to saying that the corresponding axioms are satisfied when the tensor product is replaced by the completed tensor product). The maps $\phi_x \colon U(\T) \rightarrow U(\T)$ can be continuously extended to $\overline{U(\T)}$ so that we can define again $x \cdot y = x\1 \phi_{x\2}(y)$ to obtain a structure of topological commutative automorphic Hopf algebra on $\overline{U(\T)}$. In this algebra we can define the exponential of elements in the completion $\overline{\T}$, that we can identify with the space of primitive elements of $\overline{U(\T)}$, with respect to the $I$-adic topology ($I = \T$):
\begin{displaymath}
\exp(a) := \sum_{n=0}^\infty \frac{1}{n!} a^n
\end{displaymath}
where $a \in \overline{\T}$ and  $a^n := a\cdot (\cdots (a \cdot a)) = a(\cdots (aa))$. This defines a bijection between primitive and \emph{group-like} elements of $\overline{U(\T)}$--elements $g$ with $\Delta(g) = g \otimes g$ and $\epsilon(g) = 1$--so there exist elements $\BCH(a,b)$ and $\BCH^{\cdot}(a,b)$ in $\overline{\T}$ such that
\begin{displaymath}
\exp(\BCH(a,b)) = \exp(a)\exp(b) \quad \text{and} \quad \exp(\BCH(a,b)^{\cdot}) = \exp(a) \cdot \exp(b)
\end{displaymath}
for every $a,b \in \overline{\T}$. Clearly
\begin{eqnarray*}
\exp(\BCH(a,b)^{\cdot}) &=& \exp(a) \cdot \exp(b) = \exp(a) \phi_{\exp(a)}(\exp(b))\\ 
&=& \exp(a) \exp(\phi_{\exp(a)}(b)) = \exp(\BCH(a,\phi_{\exp(a)}(b)))
\end{eqnarray*}
implies 
\begin{equation}
\label{eq:relationbch}
\BCH(a,b)^{\cdot} = \BCH(a,\phi_{\exp(a)}(b)).
\end{equation}
By Theorem \ref{thm:basis} there exist expansions
\begin{eqnarray*}
\BCH(a,b) &=& a  + b + \sum_{i,j \geq 1} \alpha_{i,j} [a,b,\stackrel{i-1}{\dots}a,\stackrel{j-1}{\dots}b] \quad  \text{and}\\
\BCH(a,b)^{\cdot} &=& a  + b + \sum_{i,j \geq 1} \beta_{i,j} [a,b,\stackrel{i-1}{\dots}a,\stackrel{j-1}{\dots}b]
\end{eqnarray*}
for some $\alpha_{i,j},\beta_{i,j} \in \field$. We will compute these coefficients by means of a concrete example of a  commutative automorphic Lie triple system. Let us consider the Lie algebra 
\begin{displaymath}
L :=\left \{ \left(\begin{array}{rrr} 0 & - \alpha & \beta \\ -\alpha & 0 & \gamma \\ 0 & 0 & 0 \end{array} \right) \mid \alpha, \beta, \gamma \in \field \right\}
\end{displaymath}
and the subspace
\begin{displaymath}
T :=\left\{ \left(\begin{array}{rrr} 0 & - \alpha & \beta \\ -\alpha & 0 & 0 \\ 0 & 0 & 0 \end{array} \right) \mid \alpha, \beta \in \field\right\}.
\end{displaymath}
The reader can check that $T$ is closed under the triple product $[a,b,c]:=\frac{1}{4}[[a,b],c]$ and it is a commutative automorphic Lie triple system (the scalar $1/4$ has been included for convenience).
The matrices 
\begin{displaymath}
u := \left(\begin{array}{rrr} 0 & 0 & 1 \\0 & 0 & 0 \\ 0 & 0 & 0 \end{array} \right), \quad \text{and}\quad v :=  \left(\begin{array}{rrr} 0 & -1 & 0 \\-1 & 0 & 0 \\ 0 & 0 & 0 \end{array} \right)
\end{displaymath}
form a basis of $T$ with 
\begin{displaymath}
[v,u,u] = 0 \quad\text{and}\quad [v,u,v] = -\frac{1}{4}u.
\end{displaymath}

\begin{lemma}
Let $\bar{a} := -2(u-v)$ and $ \bar{b} := 2(u+v)$. Then
\begin{displaymath}
[\bar{a},\bar{b},\stackrel{i-1}{\dots}\bar{a},\stackrel{j-1}{\dots}\bar{b}]  = -4 u
\end{displaymath}
if $i+j$ is odd and greater than $1$.
\end{lemma}

Let us consider the algebra $U(L)[[s,t]]$ of formal power series on two variables $s,t$ with coefficients in $U(L)$, where $U(L)$ stands for the universal enveloping algebra of the Lie algebra $L$. The associative product of $U(L)$ will be denoted by $*$. $U(T)$ and $U(L)$ are very much related. With the new product
\begin{displaymath}
xy := r(x\1)*y*r(x\2) \quad \text{where} \quad r(x\1) * r(x\2) = x
\end{displaymath}
$U(L)$ is a left Bruck Hopf algebra (see \cite{MP10a} for details). Moreover, for any $a \in T$, $r(a)1+1r(a) = a$ implies $r(a) = \frac{1}{2} a$ and $ax = r(a)*x*1 + 1*x*r(a) = \frac{1}{2}(a*x+x*a)$, thus
\begin{displaymath}
a(bc) - b(ac) = \frac{1}{4}[[a,b],c]
\end{displaymath}
i.e. $[a,b,c]$ is recovered as $a(bc)-b(ac)$. By the universal property of $U(T)$ this gives a homomorphism $U(T) \rightarrow (U(L),*)$ of unital algebras that identifies $U(T)$ and the unital subalgebra of $(U(L),*)$ generated by $T$. In the following we will assume this identification.

The homomorphism of associative algebras $U(L) \rightarrow M_3(\field)$ extends to a homomorphism 
\begin{displaymath}
\varphi \colon U(L)[[s,t]] \rightarrow M_3(\field)[[s,t]].
\end{displaymath}
In $U(L)$ we can define the exponentials $\exp(s\bar{a})$ and $\exp(t\bar{b})$ in the natural way. Under $\varphi$ we have
\begin{displaymath}
\varphi(\exp(s\bar{a}) )= e^{s\bar{a}} \quad  \varphi(\exp(t\bar{b})) = e^{t\bar{b}}
\end{displaymath}
where $e^{s\bar{a}}$ denotes the usual exponential of the matrix $s\bar{a}$.  Notice that in $U(L)[[s,t]]$ both products  $x*y$ and  $xy$ lead to the same exponential for elements of $T$. We also have a linear map (it is not a homomorphism of algebras)
\begin{displaymath}
\psi\colon \overline{U(\T)} \rightarrow U(L)[[s,t]] \stackrel{\varphi}{\rightarrow} M_3(\field)[[s,t]]
\end{displaymath}
where the first arrow is the homomorphism induced by $a\mapsto s\bar{a}$ and $b \mapsto t\bar{b}$ and the universal property of $U(\T)$. The first arrow is a homomorphism when $U(L)[[s,t]]$ is considered with the product $xy$ but in order for the second arrow to be a homomorphism we have to consider the product $x*y$ on $U(L)[[s,t]]$. Clearly
\begin{eqnarray*}
\psi(\exp(a) \exp(b)) &=& \varphi(\exp(s\bar{a})\exp(t\bar{b})) = \varphi(\exp(\frac{s}{2}\bar{a})*\exp(t\bar{b})*\exp(\frac{s}{2}\bar{a})) \\
&=& e^{\frac{s}{2}\bar{a}}e^{t\bar{b}}e^{\frac{s}{2}\bar{a}}.
\end{eqnarray*}
Thus,
\begin{displaymath}
e^{\psi(\BCH(a,b))} = e^{\frac{s}{2}\bar{a}}e^{t\bar{b}}e^{\frac{s}{2}\bar{a}}.
\end{displaymath}
After some computations we get
\begin{displaymath}
e^{\frac{s}{2}\bar{a}}e^{t\bar{b}}e^{\frac{s}{2}\bar{a}} = e^A
\end{displaymath}
with
\begin{displaymath}
A = \left(
\begin{array}{ccc}
 0 & 2 (s+t) & -\frac{2 \left(e^{2 (s+t)}-2 e^{s+2t}+2 e^s-1\right)
   (s+t)}{e^{2 (s+t)}-1} \\
 2 (s+t) & 0 & 0 \\
 0 & 0 & 0
\end{array}
\right).
\end{displaymath}
Therefore,
\begin{displaymath}
s\bar{a} + t \bar{b} -(4 \sum_{i,j\geq 1} {\alpha}_{i,j} s^i t^j )u = \left(
\begin{array}{ccc}
 0 & 2 (s+t) & -\frac{2 \left(e^{2 (s+t)}-2 e^{s+2t}+2 e^s-1\right)
   (s+t)}{e^{2 (s+t)}-1} \\
 2 (s+t) & 0 & 0 \\
 0 & 0 & 0
\end{array}
\right).
\end{displaymath}
Comparing the entry in position $(1,3)$ in both matrices we get
\begin{displaymath}
f(s,t):=-4 \sum_{i,j \geq 1} \alpha_{i,j} s^i t^j = 2s-2t-\frac{2 \left(e^{2 (s+t)}-2 e^{s+2t}+2 e^s-1\right)
   (s+t)}{e^{2 (s+t)}-1}
\end{displaymath}
By (\ref{eq:relationbch}) we have 
\begin{equation}
\psi(\BCH(a,b)^{\cdot}) = \psi(\BCH(a,\phi_{\exp(a)}(b))).
\end{equation}
To compute the right-hand side of this equality we first need to compute 
\begin{eqnarray*}
\psi(\phi_{\exp(a)}(b)) &=& t\bar{b} + \sum_{2n \geq 2} \frac{1}{(2n)!}[t\bar{b},\stackrel{2n}{\dots}s\bar{a}] = t\bar{b} + \sum_{2n\geq 2} \frac{1}{(2n)!} s^{2n}t(4u) \\
&=& t\bar{b} + t \phi(s)u
\end{eqnarray*}
with
\begin{displaymath}
\phi(s) := 4 \left(\frac{e^s+ e^{-s}}{2}-1\right).
\end{displaymath}
Hence 
\begin{eqnarray*}
-\left(4 \sum_{i,j \geq 1} \beta_{i,j} s^i t^j \right) u &=& t\phi(s)u + \\ &&  \quad \sum_{i,j\geq 1} \alpha_{i,j} [s\bar{a},t\bar{b} + t\phi(s)u, \stackrel{i-1}{\dots}s\bar{a},\stackrel{j-1}{\dots}t\bar{b} + t\phi(s)u]\\
&=& t\phi(s)u + \sum_{i,j\geq 1} \alpha_{i,j} [s\bar{a},t\bar{b}+t\phi(s)u,\stackrel{i-1}{\dots}s\bar{a},\stackrel{j-1}{\dots}t\bar{b}]\\
&=& t\phi(s)u -(4 \sum_{i,j\geq 1} \alpha_{i,j} s^it^j)u \\
&& \quad  +  \sum_{i,j\geq 1} \alpha_{i,j} s^it^j\phi(s)[\bar{a},u,\stackrel{i-1}{\dots}\bar{a},\stackrel{j-1}{\dots}\bar{b}]\\
&=& t\phi(s) u + f(s,t)u + \frac{1}{4}\phi(s) f(s,t)u
\end{eqnarray*}
so
\begin{eqnarray*}
\sum_{i,j \geq 1} \beta_{i,j} s^i t^j  &=& \frac{\left(e^{2 s}-1\right) \left(e^{2 t}+1\right) t-\left(e^{2 s}+1\right)
    \left(e^{2 t}-1\right)s}{2 \left(e^{2 (s+t)}-1\right)}\\
&=& \frac{t-s}{2}+ \frac{\left(e^{2 s}-e^{2 t}\right) (s+t)}{2 \left(e^{2 (s+t)}-1\right)}
\end{eqnarray*}

\bigskip
\begin{theorem}
\label{thm:BCH}
With the previous notation we have
\begin{displaymath}
\BCH(a,b)^{\cdot} = a + b + \sum_{i,j\geq 1} \beta_{i,j} [a,b,\stackrel{i-1}{\dots}a,\stackrel{j-1}{\dots}b]
\end{displaymath}
where $\beta_{i,j}$ ($i,j \geq 1$) is the coefficient of $s^it^j$ in the Taylor expansion of
\begin{displaymath}
 \frac{\left(e^{2 s}-e^{2 t}\right) (s+t)}{2 \left(e^{2 (s+t)}-1\right)}
\end{displaymath}
at $(0,0)$.
\end{theorem}

\begin{remark}
The coefficients $\beta_{p,q}$ were first computed with the aid of a general approach to the Baker-Campbell-Hausdorff for formal loops based on a non-associative Magnus expansion developed in \cite{MPS16}. That approach describes $\BCH(ta,b)$ as the solution of a differential equation, which provides a recursive formula for the coefficients $\alpha_{p,q}$ (notice that $\alpha_{p,q} = \beta_{p,q} = 0$ if $p+q$ is even so we assume $p+q$ is odd): 
\begin{displaymath}
p\alpha_{p,q} = \xi_{p+q-1} {\binom{p+q-2}{p-1}}  - \sum_{i,j} \xi_{p+q-i-j} \alpha_{i,j}{\binom{p+q-i-j-1}{p-i-1}}
\end{displaymath}
where the sum runs on $1 \leq i \leq p-1$, $1 \leq j \leq q$ with $i+j$ odd and $\xi_k$ are the Taylor coefficients of the expansion of  $x/\tanh(x)$ at $x=0$. Coefficients $\beta_{p,q}$ were derived as
\begin{eqnarray*}
\beta_{p,1} &=& -\frac{1}{p!} + \sum_{i \text{ even } \geq 2}^p \frac{1}{(p-i)!} \alpha_{i,1} ,\\
\beta_{p,q} &=&  \sum_{i \text{ even } \geq 2}^p \frac{1}{(p-i)!} \alpha_{i,q} \quad \text{if} \quad q >1.
\end{eqnarray*}
With the help of the computer software \emph{Mathematica} this gave the following table  $(\beta_{p,q})_{1\leq p, q\leq 7}$
\begin{displaymath}
\begin{array}{rrrrrrr}
0 & \frac{1}{3} & 0 & -\frac{1}{45} & 0 & \frac{2}{945} & 0 \\
 -\frac{1}{3} & 0 & -\frac{4}{45} & 0 & \frac{4}{315} & 0 & -\frac{8}{4725}
   \\
 0 & \frac{4}{45} & 0 & \frac{16}{945} & 0 & -\frac{64}{14175} & 0 \\
 \frac{1}{45} & 0 & -\frac{16}{945} & 0 & -\frac{16}{4725} & 0 &
   \frac{32}{22275} \\
 0 & -\frac{4}{315} & 0 & \frac{16}{4725} & 0 & \frac{128}{155925} & 0 \\
 -\frac{2}{945} & 0 & \frac{64}{14175} & 0 & -\frac{128}{155925} & 0 &
   -\frac{48896}{212837625} \\
 0 & \frac{8}{4725} & 0 & -\frac{32}{22275} & 0 & \frac{48896}{212837625} &
   0
\end{array}
\end{displaymath}
that can be checked to agree with the function in Theorem \ref{thm:BCH}.
\end{remark}

\end{document}